\documentclass[12pt]{article}
\usepackage{amsmath,amssymb}

\usepackage{color,graphicx}

\newcommand{\bs}{\boldsymbol}

\newcommand{\bu}{\bs{u}}

\newtheorem{theorem}{Theorem}

\newtheorem{proposition}{Proposition}
\newtheorem{definition}{Definition}
\newcounter{remark}
\def\theremark {\arabic{remark}}
\newenvironment{remark}{\refstepcounter{remark}\par\noindent{\bf Remark\ \theremark}\ }{\par}
\newtheorem{Proof}{Proof}

\newenvironment{proof}{\begin{Proof}\rm}{\hfill $\Box$ \end{Proof}}

\title{On Discrete Approximations to Infinite Horizon Differential Games
}
\author{ Javier de Frutos\thanks{Instituto de Investigaci\'on en Matem\'aticas (IMUVA), Universidad
de Valladolid, Spain.  (franciscojavier.frutos.baraja@uva.es)}
\and Victor Gat\'{o}n\thanks{Instituto de Investigaci\'on en Matem\'aticas (IMUVA), Universidad
de Valladolid, Spain. (victor.gaton@uva.es)}
  \and Julia Novo\thanks{Departamento de
Matem\'aticas, Universidad Aut\'onoma de Madrid, Spain. (julia.novo@uam.es)}}
\date{}

\begin{document}
\maketitle
\abstract{In this paper we study a discrete-time semidiscretization and a fully discretization (discrete-time, discrete-state)
of an infinite time horizon noncooperative $N$-player differential game. We prove that as either the discretization time step
or both time step and mesh size parameters approach zero the  discrete value function approximates the value function of the differential game. Furthermore, the discrete Nash equilibrium is an $\epsilon$-Nash equilibrium for the continuous-time differential game both in
the discrete-time and fully discrete cases.}

\section{Introduction}

 The theory of noncooperative differential games \cite{Basar}, \cite{Dockner}, \cite{Haurie_et_al}, \cite{Basar_et_al},  has become an indispensable tool in the applications to model problems in which the strategic interaction between several agents (or players) evolve over time. Among the several, non equivalent, concepts of equilibria in differential games that can be used to analyze a given model problem, we are concerned with Markovian Nash equilibria or state feedback Nash equilibria \cite{Basar_et_al}. We remark that feedback Nash equilibria have the property of being subgame perfect (strongly time consistent), see \cite{Basar_et_al}. Subgame perfectness is a property of prime importance in the applications that is not shared by other concepts of equilibrium as open-loop Nash equilibrium. It is worth noting that in  optimal control problems the optimal path can be represented by strategies either in open-loop form or in feedback form. On the contrary, when several decision-makers compete, each one faces an optimal control problem that depends on the actions of the rest of the players. Now, different information structures are not longer equivalent and, in particular, open-loop Nash equilibria are not subgame perfect \cite{Dockner}, \cite{Basar_et_al}.

  To look for a Markovian Nash equilibria each player has to solve an optimal control problem in which the strategies of his or her opponents are fixed. This leads to a system of $N$ coupled Hamilton-Jacobi-Bellman equations, $N>1$ being the number of players.
  {The combined high nonlinearity and dimensionality of Hamilton-Jacobi-Bellman equations preclude the knowledge of an analytical solution, except for some specific models with particular structure (linear-state or linear-quadratic models, for example).}
  Then we have to resort to numerical methods. In the one player case (optimal control), the numerical solution of Hamilton-Jacobi-Bellman has received considerable attention in the literature, see, among many others, the papers \cite{Falcone_Ferreti}, \cite{Falcone}, \cite{carlini}, \cite{Akian}, \cite{Guo}, \cite{Bo_et_al}.

   The objective of this paper is to show that an equilibrium of a differential game can be approximated by means of a semi-lagrangian discretization in time of the problem. Semi-lagrangian methods are well known numerical methods for optimal control problems, see for example \cite{Bardi}, \cite{Falcone}, \cite{Falcone_Ferreti}. Essentially, the method consists of a combination of time discretization of the dynamics with an approximation of the same order to the objective. This kind of methods have the nice property that once the discretization has been built up, the approximation scheme can be viewed as a discrete-time version of the continuous model. The approximation is constructed solving the Bellman equation for the discrete-time model.  The approach has been previously used, in the context of differential games, in \cite{DeFrutos2015}, \cite{DeFrutos2016}.

   In this paper we build on the results on \cite{Bardi} about the convergence of the discrete-time value function to the continuous-time value function to analyze the case of noncooperative $N$ player differential games. We prove that if the time step in the discretization is small, the discrete-time Nash equilibrium is an $\epsilon$-Nash of the differential game. Then, following \cite{Javier_yo}, we
   analyze the fully discrete (discrete-time, discrete-state) case for which we obtain analogous results to those of the discrete-time case. To this end, as in \cite{Javier_yo}, the analysis of the discrete-time discrete-state problem is based in the definition of an  auxiliary game using an appropriate interpolation on the state space.

    The rest of the paper is as follows. Section 2 is devoted to state the problem and some preliminaries including the notation to be used in the rest. In Section 3 we present the results of our analysis for the discrete time case. In Section 4
    we extend the results to the fully discrete case. Section 5 is devoted to show some numerical experiments. Finally, some concluding remarks  are presented in Section 6.

\section{Model problem and preliminaries}
We consider a $N$-player differential game with infinite time horizon.
 Player $i$'s objective, $i=1,\dots, N$, is to maximize with respect his or her own control $u_i$,
\begin{equation}
W_i(u_i,u_{-i},x_0):=\int_0^\infty f_i(x,u_i,u_{-i})e^{-\rho t} dt,\label{objective}
\end{equation}
subject to:
\begin{equation}
\dot{x}=g(x,u_i,u_{-i}),\quad x(0)=x_0.\label{dynamics}
\end{equation}

Functions $f_i:\mathbb{V}\times\mathbb{U}_1\dots\times\mathbb{U}_N\longrightarrow \mathbb{R}$, $i=1\dots, N$, and  $g:\mathbb{V}\times\mathbb{U}_1\dots\times\mathbb{U}_N\longrightarrow\mathbb{R}^n$ are given functions with
$\mathbb{V}\subset\mathbb{R}^n$ an open domain and $\mathbb{U}_i\subset\mathbb{R}^m$ a {compact}  set for $i=1,\dots, N$.
The parameter $\rho$ is a positive constant.
Here and in the rest of the paper, we are using, as it is usual, the  notation $u_{-i}$ to denote
$$u_{-i}=[u_1,\dots,u_{i-1},u_{i+1},\dots,u_N].$$
 With this notation, the evaluation of a given real function $H$ of $N$ variables in a pair $(u_i,u_{-i})$ is by convention
$$H(u_i,u_{-i})=H(u_1,\dots,u_i,\dots,u_N).$$

In the rest of this paper we will assume that functions $g$, and $f_i$, $i=1,\dots,N$, are continuous and satisfy the following assumptions:
\begin{enumerate}
\item[H1] There exists a constant $L_g$ such that for all $(x,u_1,\dots,u_N), (y,v_1,\dots,v_N)$ in $\mathbb{V}\times\mathbb{U}_1\dots\times\mathbb{U}_N$ \label{H1}
    $$
    |g(x,u_1,\dots,u_N)-g(y,v_1,\dots,v_N)|\le L_g\bigl(|x-y|+\sum_{j=1}^N|u_j-v_j|\bigr).$$
\item[H2] There exist constants $L_i$, $i=1,\dots,N$, such that
    $$
    |f_i(x,u_1,\dots,u_N)-f_i(y,v_1,\dots,v_N)|\le L_i\bigl(|x-y|+\sum_{j=1}^N|u_j-v_j|\bigr).
    $$
for all $(x,u_1,\dots,u_N), (y,v_1,\dots,v_N)$ in $\mathbb{V}\times\mathbb{U}_1\dots\times\mathbb{U}_N$ and $i=1,\dots, N$,
\item[H3] There exists a constant $M_f$ such that $(x,u_1,\dots,u_N)$ in $\mathbb{V}\times\mathbb{U}_1\dots\times\mathbb{U}_N$
$$|f_i(x,u_1,\dots,u_N)|\le M_f.$${
\item[H4] There exists a constant $M_g$ such that $(x,u_1,\dots,u_N)$ in $\mathbb{V}\times\mathbb{U}_1\dots\times\mathbb{U}_N$
$$|g(x,u_1,\dots,u_N)|\le M_g.$$}
\end{enumerate}

In this paper, we consider autonomous problems in infinite horizon and we are interested in stationary Markovian strategies, \cite{Basar}, \cite{Dockner}.
\begin{definition}\label{markovian_strategies}
Let $\mathcal{U}_i$  a set of measurable functions $\phi_i$ defined in $\mathbb{V}$ with values in  $\mathbb{U}_i\subset\mathbb{R}^m$. The set $\mathcal{U}=\mathcal{U}_1\times\dots\times\mathcal{U}_N$ is the set of admissible strategies if for every $(\phi_1,\dots,\phi_N)\in\mathcal{U}$ the state equation (\ref{dynamics}) with $u_i(t)=\phi_i(x(t))$, $i=1,\dots, N$, has, for every $x_0\in\mathbb{V}$, a unique absolutely continuous solution  $x(t)\in\mathbb{V}$ defined for all $t\ge 0$.
\end{definition}{
\begin{remark}Let us observe that the continuity of $g$ and the Lipschitz condition H1 guarantee the existence and uniqueness of the solution of the system \eqref{dynamics}.
We also assume for the domain $\mathbb{V}\subset {\Bbb R}^n$ and for the dynamics $g$ that the conditions of Definition \ref{markovian_strategies} hold so that the trajectory
$x(t)$ remains in $\mathbb{V}$. 
\end{remark}}

Given $(\psi_1\dots,\psi_N)\in\mathcal{U}$ we will use the notation
$W_{i}(\psi_i,\psi_{-i},x_0)=W_{i}({u}_i,{u}_{-i},x_0)$ with $u_{j}(t)=\psi_j(x(t))$, $j=1,\dots, N$, $t\ge 0$ and $x(t)$ defined by (\ref{dynamics}).
Let us note that, with this definition of admissible strategies,  if $(\phi_1,\dots,\phi_N)\in\mathcal{U}$ and $(\psi_1,\dots,\psi_N)\in\mathcal{U}$ are two $N$-tuples of admissible strategies, the strategy $(\psi_i,\phi_{-i})\in\mathcal{U}$ is also an admissible strategy for all $i=1,\dots,N$.

The relevant concept we are interested in is the concept of  Nash equilibrium.
\begin{definition}\label{value_function_definition}
A $N$-tuple of admissible stationary strategies $(\phi_1\dots,\phi_N)\in\mathcal{U}$ is a Markovian Nash Equilibrium (MNE) if for every $x\in\mathbb{V}$
\begin{equation}\label{nash}
W_i(\phi_i,\phi_{-i},x)\ge W_i(\psi_i,\phi_{-i},x),\quad i=1,\dots,N,
\end{equation}
for all $(\psi_1,\dots,\psi_N)\in\mathcal{U}$.

Given a MNE $(\phi_1,\dots,\phi_N)$ the value function for player $i$ is the function
$$V_i(x)=W_i(\phi_i,\phi_{-i},x),\quad x\in\mathbb{V}.$$
\end{definition}	
The following verification theorem can be found in \cite[Theorem 4.1]{Dockner}
\begin{theorem}\label{verification}
Let $(\phi_1,\dots,\phi_N)\in\mathcal{U}$ a $N$-tuple of admissible stationary strategies. Assume that { the functioms $V_i:\mathbb{V}\rightarrow\mathbb{R}$, $i=1,\dots,N$ are viscosity solutions of the Hamilton-Jacobi-Bellman equations, \cite{wagener}}
\begin{equation}\label{HJB}
\rho V_i(x)=\max_{u_i\in\mathbb{U}_i}\left\{f_i(x,u_i,\phi_{-i})+\nabla V_i(x)^Tg(x,u_i,\phi_{-i}))\right\}, \quad i=1,\dots,N,
\end{equation}
for all  $x\in\mathbb{V}$.
Assume also that either $V_i$ is bounded or $V_i$ is bounded below and the transversality condition
\begin{equation}\label{transversality}
\limsup_{T\rightarrow\infty}e^{-\rho T}V_i(x(T)) \le 0,
\end{equation}
where $x(t)$ is the solution of  \eqref{dynamics} with $u_i(t)=\phi_i(x(t))$, $i=1,\dots, N$, is satisfied.
If $\phi_i(x)$ is a maximizer of the right hand side of\eqref{HJB} for all $i=1,\dots,N$ and $x\in\mathbb{V}$, then $(\phi_1,\dots,\phi_N)$ is a Markovian Nash Equilibrium. Moreover, the function $V_i$ is the value function for player $i$, $i=1,\dots, N$.
\end{theorem}
{The Markovian Nash Equilibrium is in the sense of catching up optimality \cite{Carlson}}.
We remark that (\ref{HJB}) is a non-linear partial differential equation whose solution requires of some numerical approximation, except for some particular cases as linear-state or linear-quadratic problems, for example.
In the case of optimal control problems (only one player) one well developed approach is to combine a time discretization of (\ref{dynamics}) with a discretization of the same order of (\ref{objective}), see \cite{Bardi}, \cite{Falcone_Ferreti}, \cite{Falcone}, for example. For differential games (more than one interacting player) this approach has been used in \cite{DeFrutos2015}, \cite{DeFrutos2016}.
We consider now the most simple time-discrete version of the problem (\ref{objective})-(\ref{dynamics}). We consider a
discretization of the functional (\ref{objective}) by means of the rectangle rule combined with a forward Euler discretization of the dynamics (\ref{dynamics}).

Let $h>0$ be a positive parameter and let $t_n =nh$ be the
discrete times defined for all positive integers $n$. We denote by $\beta_h$ the discrete discount factor defined by $\beta_h=1-\rho h$. We consider the discrete-time infinite horizon game in which player $i$ aims to maximize

\begin{equation}
W_{i,h}(\bs{u}_i,\bs{u}_{-i},x_0):=h\sum_{n=0}^\infty \beta_h^n f_i(x_n,u_{i,n},u_{-i,n}), \label{objective_discrete}
\end{equation}
subject to
\begin{equation}\label{dynamics_discrete}
x_{n+1}=x_n+hg(x_n,u_{i,n},u_{-i,n}),
\end{equation}
where
$$
\bs{u}_j=\left\{u_{j,0},u_{j,1},\ldots\right\},\quad u_{j,n}\in\mathbb{U}_j,\quad n\ge 0,\quad j=1,\dots, N,
$$
and $x_0\in\mathbb{V}$ is a given initial state.

We are interested in stationary Markovian Strategies, see \cite{Haurie_et_al}, \cite{Krawczyk_Petkrov} for a study of the discrete-time  case. We assume, for simplicity, that for every $x_0\in \mathbb{V}$ and $(\psi_1,\dots,\psi_N)\in\mathcal{U}$,
the recursion (\ref{dynamics_discrete}) with $u_{j,n}=\psi_j(x_n)$, $j=1,\dots, N$, $n\ge 0$, is well defined and $x_n\in\mathbb{V}$ for all $n\ge 0$. In other words, we assume that $\mathcal{U}$ is also the set of admissible strategies of the  discrete-time game (\ref{objective_discrete}), (\ref{dynamics_discrete}). Let $(\psi_1,\dots,\psi_N)\in\mathcal{U}$. We will use the notation
$W_{i,h}(\psi_i,\psi_{-i},x_0)=W_{i,h}(\bs{u}_i,\bs{u}_{-i},x_0)$ with $u_{j,n}=\psi_j(x_n)$, $j=1,\dots, N$, $n\ge 0$ and $x_n$ defined by the recursion (\ref{dynamics_discrete}).

The definitions of Markovian Nash Equilibrium and player $i$ value function are similar to that of the continuous-time dynamic game.

\begin{definition}\label{value_function_definition_time_discrete}
A $N$-tuple of admissible stationary strategies $(\phi_1^h\dots,\phi_N^h)\in\mathcal{U}$ is a Markovian Nash Equilibrium (MNE) for the discrete-time game \eqref{objective_discrete}-\eqref{dynamics_discrete} if for every $x\in\mathbb{V}$
\begin{equation}\label{nash_time_discrete}
W_{i,h}(\phi_i^h,\phi_{-i}^h,x)\ge W_{i,h}(\psi_i,\phi_{-i}^h,x),\quad i=1,\dots,N.
\end{equation}
for all $(\psi_1,\dots,\psi_N)\in\mathcal{U}$.

Given a MNE for the discrete-time game $(\phi_1^h,\dots,\phi_N^h)$ the value function for player $i$ is the function
$$V_{i,h}(x)=W_{i,h}(\phi_i^h,\phi_{-i}^h,x),\quad x\in\mathbb{V}.$$
\end{definition}
The following is a verification theorem similar to (\ref{verification}), see \cite{Krawczyk_Petkrov}, \cite{Haurie_et_al}.

\begin{theorem}\label{verification_time_discrete}
Let $(\phi_1^h,\dots,\phi_N^h)\in\mathcal{U}$ a $N$-tuple of admissible stationary strategies. Assume that there exist continuous functions $V_{i,h}:\mathbb{V}\rightarrow\mathbb{R}$, $i=1,\dots,N$, such that the Bellman equations
\begin{equation}\label{Bellman}
 V_{i,h}(x)=\max_{u_i\in\mathbb{U}_i}\left\{hf_i(x,u_i,\phi_{-i}^h)+\beta_h V_{i,h}(x+hg(x,u_i,\phi_{-i}^h))\right\}, \quad i=1,\dots,N,
\end{equation}
are satisfied for all  $x\in\mathbb{V}$.
Assume also that either $V_{i,h}$ is bounded or $V_{i,h}$ is bounded below and the transversality condition
\begin{equation}\label{discrete_time_transversality}
\limsup_{n\rightarrow\infty}\beta_h^nV_{i,h}(x_n) \le 0,
\end{equation}
where $\{x_n\}_{n=0}^\infty$ is the solution of  \eqref{dynamics_discrete} with $u_{i,n}=\phi_i^h(x_n)$, $i=1,\dots, N$, is satisfied.
If $\phi_i^h(x)$ is a maximizer of the right hand side of\eqref{Bellman} for all $i=1,\dots,N$ and $x\in\mathbb{V}$, then $(\phi_1^h,\dots,\phi_N^h)$ is a Markovian Nash Equilibrium for the discrete game. Moreover, the function $V_{i,h}$ is the value function for player $i$, $i=1,\dots, N$.
\end{theorem}{
\begin{remark}\label{remark2}
We observe that to compute the solution of \eqref{nash_time_discrete}, or equivalently the solution of \eqref{Bellman}, one fixes all the strategies except for one to get the maximum.
This maximum exists for $h$ small enough, being the fixed point of a contractive operator, see \cite{Capuzzo_paper}.
\end{remark}}
\section{Discrete-time approximation analysis}

The following proposition is a consistency result that extends \cite[Chapter 6, Lemma 1.2]{Bardi} to the case of a number of players $N>1$. We include the proof for the reader's convenience.
\begin{proposition}\label{consistency}
Let  $(\phi_1,\dots,\phi_N)\in\mathcal{U}$ an arbitrary $N$-tuple of admissible strategies. Let us assume that there exists a constant $L_s>0$ with
\begin{equation}\label{lipschitz_strategies}
|\phi_i(x_1)-\phi_i(x_2)|\le L_s |x_1-x_2|,\quad i=1,\dots, N.
\end{equation}
Let us assume that hypotheses $H_1$, $H_2$ and $H_3$ are satisfied.
 Then
\begin{equation*}
\lim_{h\rightarrow 0} \left|W_{i,h}(\phi_{i},\phi_{-i},x)- W_i(\phi_{i},\phi_{-i},x)\right|=0,
\quad i=1,\dots,N.
\end{equation*}
\end{proposition}
\begin{proof} The first part of the proof uses a well known argument from the theory of the numerical solution of ordinary differential equations, see \cite[Chapter 6, Lemma 1.2]{Bardi} .

Let $(\phi_1,\dots,\phi_N)\in\mathcal{U}$ be a fixed $N$-tuple of admissible strategies.
Let $y(t)$ be the solution of (\ref{dynamics}) with $y(0)=x$,
and $y_n$, $n=0,1,\dots$ the solution of (\ref{dynamics_discrete}) with $y_0=x$. Let us define the piecewise constant function
$$\tilde{y}(t)=y_n,\quad t\in [t_n, t_{n+1}), \quad n\ge 0,$$
with $t_n=nh$, $n=0,1,\dots$.

Let us note that $\phi_i(\tilde{y}(t))$ is a piecewise constant strategy with $\phi_i(\tilde{y}(t))=\phi_i(y_n)$ for $t\in [t_n, t_{n+1})$.

It is easy to see that $\tilde{y}$ can be expressed as
$$
\tilde{y}(t)=x+\int_0^{t_n}g(\tilde{y}(s),\phi_i(\tilde{y}(s)),\phi_{-i}(\tilde{y}(s))\,ds,
\quad \forall t\in[t_n, t_{n+1}).
$$
Using that $y(t)$ satisfies
$$
{y}(t)=x+\int_0^{t}g({y}(s),\phi_i(y(s)),\phi_{-i}(y(s))\, ds, \quad  \forall t\ge 0,
$$
we have that, for $t\in[t_n,t_{n+1})$,
\begin{align}\label{eq_prime_pro}
y(t)-\tilde y(t)&= \int_0^{t_n}g(y(s),\phi_i(y(s)),\phi_{-i}(y(s)))
-g(\tilde y(s),\phi_i(\tilde{y}(s)),\phi_{-i}(\tilde{y}(s))\, ds\nonumber\\
&\quad+\int_{t_n}^t g(y(s),\phi_i(y(s)),\phi_{-i}(y(s))) \, ds.
\end{align}

 We use now hypothesis H1, (\ref{lipschitz_strategies})  and the fact that H1 implies that there exists a constant
 $K\ge 0$ with
\begin{equation}\label{bound_dynamics}
|g(y,u_{i},u_{-i})|\le K\big(1+|y|\bigr),
\end{equation}
to get
\begin{equation*}
|y(t)-\tilde y(t)|\le  L\int_0^t |y(s)-\tilde y(s)|~ds+K\int_{[t/h]h}^t(1+|y(s)|) \, ds,
\end{equation*}
where $ L=L_g(1+NL_s)$.

It is easy to prove (see \cite[Chapter 3, Theorem 5.5]{Bardi}) that, thanks to (\ref{bound_dynamics})
\begin{equation}\label{eq:cota_sol_or}
|y(t)|\le (|x|+\sqrt{2Kt})e^{K t},\quad t>0.
\end{equation}
And then
\begin{equation*}
|y(t)-\tilde y(t)|\le   L\int_0^t |y(s)-\tilde y(s)|~ds+K h\left(1+(|x|+\sqrt{2Kt})e^{Kt}\right).
\end{equation*}
Hence, by Gronwall's Lemma
\begin{equation}\label{estimacion_dynamics}
|y(t)-\tilde y(t)|\le K h\left(1+(|x|+\sqrt{2Kt})e^{Kt}\right)e^{Lt}.
\end{equation}

Let $J$ be a positive integer. Let us write
\begin{equation*}
\begin{split}
W_{i,h}(\phi_{i},\phi_{-i},x) &=
h\sum_{n=0}^{J-1}\beta_h^nf_i(y_n,\phi_i(y_n),\phi_{-i}(y_n))
+h\sum_{n=J}^\infty\beta_h^nf_i(y_n,\phi_i(y_n),\phi_{-i}(y_n))\\
&=\int_0^{t_J}\beta_h^{[s/h]h}f_i(\tilde{y}(s),\phi_i(\tilde{y}(s)),\phi_{-i}(\tilde{y}(s)))\, ds\\
&\qquad+h\sum_{n=J}^\infty\beta_h^nf_i(y_n,\phi_i(y_n),\phi_{-i}(y_n)).
\end{split}
\end{equation*}

So that

\begin{equation}\label{estimation_functional}
\begin{split}
W_{i,h}(\phi_{i},\phi_{-i}&,x)-W_{i}(\phi_{i},\phi_{-i},x)=\\
&\int_0^{t_J}e^{-\rho s}\bigl(f_i(\tilde{y}(s),\phi_i(\tilde{y}(s)),\phi_{-i}(\tilde{y}(s)))
-f_i({y}(s),\phi_i({y}(s)),\phi_{-i}({y}(s))) \bigr )\, ds\\
&+\int_0^{t_J}(\beta_h^{[s/h]}-e^{-\rho s})f_i(\tilde{y}(s),\phi_i(\tilde{y}(s)),\phi_{-i}(\tilde{y}(s)))\, ds\\
&+h\sum_{n=J}^\infty\beta_h^nf_i(y_n,\phi_i(y_n),\phi_{-i}(y_n))\\
&+\int_{t_J}^\infty e^{-\rho s}f_i({y}(s),\phi_i({y}(s)),\phi_{-i}({y}(s))) \, ds.
\end{split}
\end{equation}
We bound separately each of the four terms. Using now H2, (\ref{lipschitz_strategies}) and (\ref{estimacion_dynamics}),  we have
\begin{equation*}
\Bigl |\int_0^{t_J}e^{-\rho t}\Bigl(f_i(\tilde{y}(s),\phi_i(\tilde{y}(s)),\phi_{-i}(\tilde{y}(s)))
-f_i({y}(s),\phi_i({y}(s)),\phi_{-i}({y}(s))) \Bigr )\, ds\Bigr |\le Ch,
\end{equation*}
with $C=KL_i(1+NL_s))\int_0^{t_J} \bigl(1+(|x|+\sqrt{2Ks})e^{Ks}\bigr)e^{Ls}e^{-\rho s} \, ds.$
The second term can be estimated using H3 and the mean value theorem as follows
\begin{equation*}
\begin{split}
\bigl
|\int_0^{t_J}(\beta_h^{[t/h]}-e^{-\rho t})f_i(\tilde{y}(s),\phi_i(\tilde{y}(s)),\phi_{-i}(\tilde{y}(s)))\, ds\bigr|&\le
Mt_J\max_{0\le t\le t_J} |\beta_h^{[t/h]}-e^{-\rho t}|\\
&\le M\rho t_J((\theta_h-1)t_J+\theta_hh),
\end{split}
\end{equation*}
with $\theta_h=-\log(1-\rho h)/(\rho h)$. Note that $\theta_h\rightarrow 1$ as $h\rightarrow 0$.
Third and four terms in (\ref{estimation_functional}) are bounded using hypothesis H3. We have
$$
h\bigl |\sum_{n=J}^\infty\beta_h^nf_i(y_n,\phi_i(y_n),\phi_{-i}(y_n))\bigr|\le
M\frac{\beta_h^J}{1-\beta_h}h
$$
and
$$\bigl |\int_{t_J}^\infty e^{-\rho t}f_i({y}(s),\phi_i({y}(s)),\phi_{-i}({y}(s)))\, ds\bigr|\le
\frac{M}{\rho}e^{-\rho t_J}.
$$
The proof finishes observing that each of the four terms can be made arbitrary small taking  $J$ big enough and $h$ small enough. \end{proof}

The following proposition is a refinement of the Proposition \ref{consistency} requiring stronger hypotheses on the problem data {\cite[Chapter 6 Lemma 1.2]{Bardi}}.

\begin{proposition}\label{consitency_with_order}
Let  $(\phi_1,\dots,\phi_N)\in\mathcal{U}$ an arbitrary $N$-tuple of {(time-continuous)} admissible strategies and let $x\in\mathbb{V}$. Let assume that hypotheses $H1$, $H2$, $H3$ and \eqref{lipschitz_strategies} hold.
Let us assume that either {$\rho> L$} with $L=L_g(1+LN_s)$ and $H4$ holds or {$\rho>L+K$} with $K$ the constant in \eqref{bound_dynamics}.  Then, there exists a
positive constant $C$ and $h_0>0$ such that for all $h\le h_0$ {with $h_0<1/\rho$}
\begin{equation*}
\bigl|W_{i,h}(\phi_i,\phi_{-i},x)-W_i(\phi_i,\phi_{-i},x)\bigr |\le C h.
\end{equation*}
\end{proposition}
\begin{proof}
In the proof we will use the same notation as in Proposition \ref{consistency}.

We start by writing
\begin{align} \label{auxiliar_consitency_with_order}
\Bigl |W_{i,h}({\phi}_i,{\phi}_{-i}&,x)- W_i(\phi_i,\phi_{-i},x)\Bigr |\le\nonumber\\
&\int_0^\infty |f_i(y(t),\phi_i(y(t)),\phi_{-i}(y(t)))-f_i(\tilde{y}(t),{\phi}_i(\tilde{y}(t)),\phi_{-i}(\tilde{y}(t)))|e^{-\rho t}\, ds,\nonumber\\
&+\int_0^\infty|f_i(\tilde{y}(t),{\phi}_i(\tilde{y}(t)),\phi_{-i}(\tilde{y}(t)))||e^{-\rho t}-e^{- \rho\theta_h[t/h]h}|\, dt,
\end{align}
where $\theta_h=-\log(1-\rho h)/(\rho h)$.
Then, using  H2, (\ref{lipschitz_strategies}) and (\ref{estimacion_dynamics}) we have, if {$\rho>K+ L$},
\begin{equation*}
\begin{split}
\int_0^\infty |f_i(y(t),\phi_i(y(t)),&\phi_{-i}(y(t)))-f_i(\tilde{y}(t),{\phi}_i(\tilde{y}(t)),\phi_{-i}(\tilde{y}(t)))|e^{-\rho t}\, dt\le \\
&hKL_i(1+NL_s)\int_0^\infty\left(1+(|x|+\sqrt{2Kt})e^{Kt}\right)e^{(L-\rho)t} \, dt \le C h,
\end{split}
\end{equation*}
for some constant $C>0$.

The second term in (\ref{auxiliar_consitency_with_order}) can be estimated using hypothesis H3 and the mean
 value theorem {applied to the function $e^{-\rho s}$}

\begin{equation*}
\begin{split}
\int_0^\infty|f_i(\tilde{y}(t),{\phi}_i(\tilde{y}(t)),\phi_{-i}(\tilde{y}(t)))|&|e^{-\rho t}-e^{- \rho\theta_h[t/h]h}|\, dt\le\\&{
M_f \int_0^\infty \rho\max\left\{e^{-\rho t},e^{-\rho\theta_h[t/h]h}\right\}\rho| t- \theta[t/h]h| \, dt.}
\end{split}
\end{equation*}
{Taking into account that $[t/h]h\le t\le [t/h]h+h$ and $\theta_h>1$ then $|t- \theta[t/h]h|\le (\theta_h-1)t+\theta_h h$. On the other hand, it is easy to check that $\max\left\{e^{-\rho t},e^{-\rho\theta_h[t/h]h}\right\}
\le e^{-\rho t} e^{\rho \theta_h}.$ Then}
\begin{equation*}
\begin{split}
\int_0^\infty|f_i(\tilde{y}(t),{\phi}_i(\tilde{y}(t)),\phi_{-i}(\tilde{y}(t)))|&|e^{-\rho t}-e^{- \rho\theta_h[t/h]h}|\, dt\le\\
&M_f\rho^2 {e^{\theta_h\rho }}\int_0^{\infty}e^{-\rho t}((\theta_h-1)t+\theta_h h) \, dt,\\
\end{split}
\end{equation*}
and since
 $ \theta_h-1=\mathcal{O}(h)$ as $h\rightarrow 0$, we conclude
\begin{equation*}
\int_0^\infty|f_i(\tilde{y}(t),{\phi}_i(\tilde{y}(t)),\phi_{-i}(\tilde{y}(t)))||e^{-\rho t}-e^{- \rho\theta_h[t/h]h}|\, dt\le Ch,
\end{equation*}
for some constant $C>0.$

Let us assume $H4$ and {$\rho>L$}.
We have immediately that{
$$
|y(t)-\tilde y(t)|\le M_g h e^{Lt},
$$
}
and  then
\begin{equation*}
\begin{split}
\int_0^\infty |f_i(y(t),\phi_i(y(t)),\phi_{-i}(y(t)))-f_i(\tilde{y}(t)&,{\phi}_i(\tilde{y}(t)),\phi_{-i}(\tilde{y}(t)))|e^{-\rho t}\, ds\le \\
&L_i(1+NL_s)M_gh\int_0^\infty e^{(L-\rho)t}\, dt \le Ch,
\end{split}
\end{equation*}
for some positive constant $C>0$ which finishes the proof.
\end{proof}
The following theorem is one of the main objectives of this paper. It states that a Markov Nash equilibrium of the discrete-time game is an approximate Nash equilibrium for the differential game in the sense that for $\epsilon>0$ arbitrary it constitutes an $\epsilon$-Nash equilibrium for $h$ small enough.
\begin{theorem}\label{epsilon-Nash}
 Let $(\phi_1^h,\dots,\phi_N^h)$  a Markov Nash equilibrium of the discrete time game
\eqref{objective_discrete}-\eqref{dynamics_discrete} that satisfies \eqref{lipschitz_strategies}. Let us assume that hypotheses $H1$, $H2$ and $H3$ are satisfied. Let  $\epsilon>0$. There exists $h_0>0$ such that for $h\le h_0$  and all $x\in\mathbb{V}$, if $(\psi_1,\dots,\psi_n)\in\mathcal{U}$ is a $N$-tuple of arbitrary admissible stationary strategies satisfying \eqref{lipschitz_strategies}, then
\begin{equation*}
W_i(\phi_i^h,\phi_{-i}^h,x)\ge W_i(\psi_i,\phi_{-i}^h,x)-\epsilon,\quad i=1,\dots, N.
\end{equation*}
\end{theorem}
\begin{proof}
From Proposition \ref{consistency} we known that given $\epsilon>0$ there exists a constant $h_0$ such that for every $h\le h_0$ and every $N$-tuple $(\varphi_1,\dots,\varphi_N)\in\mathcal{U}$ satisfying (\ref{lipschitz_strategies})
$$\bigl|W_{i,h}(\varphi_i,\varphi_{-i},x)-W_i(\varphi_i,\varphi_{-i},x)\bigr|\le \frac{\epsilon}{2}.$$

Using (\ref{nash_time_discrete}) we get
\begin{align*}
W_{i}(\phi_i^h,\phi_{-i}^h,x)
&=\bigl (W_{i}(\phi_i^h,\phi_{-i}^h,x)-W_{i,h}({\phi}_i^h,{\phi}_{-i}^h,x)\bigr )
+W_{i,h}({\phi}_i^h,{\phi}_{-i}^h,x)\\
&\ge\bigl(W_{i}(\phi_i^h,\phi_{-i}^h,x)-W_{i,h}({\phi}_i^h,{\phi}_{-i}^h,x)\bigr)
+W_{i,h}(\psi_i,{\phi}_{-i}^h,x)\\
 &=W_{i}(\psi_i,\phi_{-i}^h,x)
 +\bigl (W_{i}(\phi_i^h,\phi_{-i}^h,x)-W_{i,h}({\phi}_i^h,{\phi}_{-i}^h,x)\bigr)\\
 &\phantom{=W_{i}(\psi_i,\phi_{-i}^h,x))} + \bigl(W_{i,h}({\psi}_i,{\phi}_{-i}^h,x)-W_{i}(\psi_i,\phi_{-i}^h,x)\bigr).
\end{align*}
Then, noting that all the bounds in the proof of  Proposition \ref{consistency} depend only on $H1$, $H2$, $H3$ and the constant $L_s$, we have that for $h$ small enough
\begin{equation*}
W_{i}(\phi_i^h,\phi_{-i}^h,x)
\ge W_{i}(\psi_i,\phi_{-i}^h,x)-\epsilon.
\end{equation*}
\end{proof}
Next theorem is a refinement of Theorem \ref{epsilon-Nash} with the more exigent hypotheses of Proposition \ref{consitency_with_order}.
\begin{theorem}\label{th4}
Let $(\phi_1^h,\dots,\phi_N^h)$  a Markov Nash equilibrium of the discrete time game
\eqref{objective_discrete}-\eqref{dynamics_discrete} that satisfy \eqref{lipschitz_strategies}. Let us assume that hypothesis $H1$, $H2$ and $H3$ are satisfied. Furthermore, let assume that either  $\rho> L$, with $L=L_g(1+NL_s)$ and $H4$ holds or {$\rho>L_g+K$} with $K$ the constant in \eqref{bound_dynamics}.
 There exists a
positive $C>0$ and $h_0>0$ such that for all $h\le h_0$ and all $x\in\mathbb{V}$, if $(\psi_1,\dots,\psi_n)\in\mathcal{U}$ is a $N$-tuple of arbitrary admissible stationary strategies satisfying \eqref{lipschitz_strategies}, then
\begin{equation*}
W_{i}(\phi_i^h,\phi_{-i}^h,x)\ge W_{i}(\psi_i,\phi_{-i}^h,x)-C h.
\end{equation*}
\end{theorem}
\begin{proof} The proof is  exactly the same as in Theorem~\ref{epsilon-Nash} using now Proposition~\ref{consitency_with_order} instead of Proposition~\ref{consistency}.
\end{proof}
\section{Fully discrete case}

Let $ \Omega\subseteq\mathbb{V}$ be a bounded polyhedron in $\mathbb R^n$ such that for sufficiently small $h>0$ {the following inward pointing condition on the dynamics} holds
\begin{equation}\label{invariance}
x+hg(x,u_i,u_{-i})\in \overline \Omega,\quad \forall x \in \overline \Omega, \ u_i\in \mathbb{U}_i.
\end{equation}
Let $\left\{S_j\right\}_{j=1}^{m_s}$ be a family of simplices which defines a regular triangulation of $\Omega$
$$
\overline \Omega=\bigcup_{j=1}^{m_s} S_j,
$$
and let $ k=\max_{1\le j\le m_s}({\rm diam} \ S_j)$. We assume we have $n_s$ vertices (nodes), denoted $x^1,\ldots,x^{n_s}$, in the triangulation. Let $V^k$ be the space of piecewise affine functions from $\overline \Omega$ to ${\mathbb R}$ which are continuous in $\overline \Omega$ having constant gradients in the interior of any simplex $S_j$ of the triangulation.
\begin{definition}\label{definition_piecewise_value_approximation}
Let $(\phi_1^j,\ldots,\phi_N^j)$, $j=1,\ldots,n_s$, with $\phi_i^j\in {\mathbb U}_i$, $i=1,\dots, N$, $j=1,\ldots,n_s$.
Assume that there exist continuous piecewise affine functions $V_{i,h,k}\in V^k$, $i=1,\dots,N$, such that
\begin{equation}\label{fully-Bellman}
 V_{i,h,k}(x^j)=\max_{u_i^j\in\mathbb{U}_i}\left\{hf_i(x^j,u_{i}^j,\phi_{-i}^j)+\beta_h V_{i,h,k}(x^j+hg(x^j,u_{i}^j,\phi_{-i}^j))\right\},
\end{equation}
for any vertex,  $x^j\in \overline \Omega$, $j=1,\ldots,n_s$, and that $\phi_i^j$  is a maximizer of the right hand side of \eqref{fully-Bellman} for every $i=1,\dots,N$, $j=1,\dots n_s$.
Then, the function $V_{i,h,k}\in V^k$ defined
by its nodal values in \eqref{fully-Bellman} is the fully discrete approximation to the value function for player $i$, $i=1,\dots, N$.
\end{definition}
{To compute the numerical approximations \eqref{fully-Bellman} one fixes all the strategies except for one
to compute the maximum. A contraction argument analogous to that of the semi-discrete case (see Remark \ref{remark2}) can be applied to
guarantee existence and uniqueness of that maximum, see also \cite[Theorem 1.1, Appendix A]{Bardi}}.

Let us define $\phi_i^{h,k}\in V^k$ as the piecewise affine function determined by
\begin{equation}\label{piecewise_linear_epsilon_Nash}
\phi_i^{h,k}(x^j)=\phi_i^j.
\end{equation}
The rest of this section is devoted to prove that the strategies $(\phi_1^{h,k},\dots,\phi_N^{h,k})$ are an $\epsilon$-Nash for the differential game \eqref{objective}-\eqref{dynamics}.

Next, we define an auxiliary time-discrete game such that its value functions coincide with $V_{i,h,k}$.

\begin{definition}
Let $(\phi_1,\dots,\phi_N)\in \mathcal{U}$ admissible strategies. Let us define
\begin{equation}
W_{i,h,k}(\phi_i,\phi_{-i},x_0):=h\sum_{n=0}^\infty \beta_h^n I_kf_i(x_n,\phi_i(x_n),\phi_{-i}(x_n)), \label{fully-objective_discrete}
\end{equation}
subject to
\begin{equation}\label{fully-dynamics_discrete}
x_{n+1}=x_n+hI_kg(x_n,\phi_i(x_n),\phi_{-i}(x_n)),
\end{equation}
where
\begin{eqnarray}\label{defino1b}
I_kg(x,\phi_i(x),\phi_{-i}(x))&=&\sum_{j=1}^{n_s}\mu_j(x)g(x^j,\phi_i(x^j),\phi_{-i}(x^j)),\\
I_kf_i(x,\phi_i(x),\phi_{-i}(x))&=&\sum_{j=1}^{n_s}\mu_j(x)f_i(x^j,\phi_i(x^j),\phi_{-i}(x^j)).\label{defino2b}
\end{eqnarray}
Here, $\mu_j(x)$, $j=1,\dots,n_s$ denote the barycentric coordinates of $x\in\overline \Omega$ with respect to the triangulation $\{S_j\}_{j=1}^{m_s}$.
\end{definition}
We recall that the barycentric coordinates of $x \in\overline\Omega$ with respect to the triangulation $\{S_j\}_{j=1}^{m_s}$ is the set of real numbers $\mu_j(x)$, $j=1,\dots,n_s$ defined by
$$x=\sum_{j=1}^{n_s}\mu_j(x)x^j,\quad 0\le\mu_j(x)\le 1,\quad \sum_{j=1}^{n_s}\mu_j(x)=1,$$
where $\{x^j\}_{j=1}^{n_s}$ are the nodes of the partition.

Let us observe that a Markov Nash Equilibrium for the time-discrete game \eqref{fully-objective_discrete}-\eqref{fully-dynamics_discrete} is defined only by its values at the nodes $x^j$, $j=1,\ldots,n_s$, (see \eqref{defino1b}, \eqref{defino2b}). Using this observation and arguing as in \cite[Theorem 3]{Javier_yo}, we can prove the following theorem that states that the functions $V_{i,h,k}$  in \eqref{fully-Bellman} are, in fact the value functions for \eqref{fully-objective_discrete}-\eqref{fully-dynamics_discrete}.

\begin{theorem}\label{verification_fully-time_discrete}
The set of piecewise affine strategies $(\phi_1^{h,k},\ldots,\phi_N^{h,k})$ defined in \eqref{piecewise_linear_epsilon_Nash} is a Markov Perfect Nash equilibrium for \eqref{fully-objective_discrete}-\eqref{fully-dynamics_discrete}. Moreover, the functions defined in \eqref{fully-Bellman} satisfy
$$V_{i,h,k}(x)=W_{i,h,k}(\phi_i^{h,k},\phi_{-i}^{h,k},x),\quad x\in\Omega,$$
\end{theorem}
For the proof of the main results of this section we introduce a discrete auxiliary  function in the following definition. This
function is compared with $W_i$  (see \eqref{objective}) in Propositions \ref{consistency_fully} and \ref{consitency_with_order_fully} below.
{In the following definition a new interpolant operator, $\tilde I_k,$ is introduced for which we apply interpolation respect to the first argument while the rest of the arguments behave as parameters.}
\begin{definition} Let $\bs{u}_i$, $i=1,\dots, N$ such that
$$
\bs{u}_i=\left\{u_{i,0},u_{i,1},\ldots\right\},\quad u_{i,n}\in\mathbb{U}_i,\quad n\ge 0,\quad i=1,\dots, N,
$$
and let us denote by $\mu_j(x)$, $j=1,\dots, n_s$, $x\in \overline\Omega$, the barycentric coordinates with respect the partition $\{S_j\}_{j=1}^{m_s}$.
Then,
\begin{equation}
\widetilde W_{i,h,k}(\bu_i,\bu_{-i},x_0):=h\sum_{n=0}^\infty \beta_h^n \tilde{I}_kf_i(x_n,u_{i,n},u_{-i,n}), \label{fully-objective_discrete_aux}
\end{equation}
subject to
\begin{equation}\label{fully-dynamics_discrete_aux}
x_{n+1}=x_n+h\tilde{I}_kg(x_n,u_{i,n},u_{-i,n}),
\end{equation}
where
\begin{eqnarray*}\label{defino1}
\tilde{I}_kg(x_n,u_{i,n},u_{-i,n})&=&\sum_{j=1}^{n_s}\mu_j(x_n)g(x^j,u_{i,n},u_{-i,n}),\\
\tilde{I}_kf_i(x_n,u_{i,n},u_{-i,n})&=&\sum_{j=1}^{n_s}\mu_j(x_n)f_i(x^j,u_{i,n},u_{-i,n}).\label{defino2}
\end{eqnarray*}
Given $(\phi_1,\dots,\phi_N)\in\mathcal{U}$ an $N$-tuple of admissible strategies we also define
$$
\widetilde W_{i,h,k}(\phi_{i},\phi_{-i},{x_0})=\widetilde W_{i,h,k}(\bu_i,\bu_{-i},x_0),
$$
where $u_{j,n}=\phi_j(x_n)$ and $x_n$ is defined in \eqref{fully-dynamics_discrete_aux}.
\end{definition}
The following proposition is the analogous to Proposition \ref{consistency} for the fully discrete case.
\begin{proposition}\label{consistency_fully}
Let  $(\phi_1,\dots,\phi_N)\in\mathcal{U}$ be an $N$-tuple of admissible strategies.  Then
\begin{equation*}
\lim_{h\rightarrow 0, k\rightarrow 0} \left|\widetilde W_{i,h,k}(\phi_{i},\phi_{-i},x)- W_i(\phi_{i},\phi_{-i},x)\right|=0,
\quad i=1,\dots,N.
\end{equation*}
\end{proposition}
\begin{proof}
The proof follows the arguments of Proposition \ref{consistency}  with the technique of the proof of \cite[Lemma 1]{Javier_yo}
to deal with the extra terms coming from the interpolation error.
\end{proof}
The following proposition is analogous to Proposition \ref {consitency_with_order}
\begin{proposition}\label{consitency_with_order_fully}
Assume conditions of Proposition 3 hold.  Assume also that {$\rho> L$}, with $L=L_g(1+NL_s)$.  Then, there exist
positive constants $C$ and $h_0>0$ such that for all $h\le h_0$
\begin{equation*}
\left|\widetilde W_{i,h,k}(\phi_{i},\phi_{-i},x)- W_i(\phi_{i},\phi_{-i},x)\right|\le C (h+k).
\end{equation*}
\end{proposition}
\begin{proof}
Since we are assuming H4 ($g$ bounded), the proof of Proposition \ref{consitency_with_order_fully} can be
obtained arguing as in Proposition \ref{consitency_with_order} for the case in which $g$ is bounded. As before, we also argue as in
\cite[Lemma 2]{Javier_yo}, to deal with the interpolation errors.
\end{proof}
\begin{theorem}\label{epsilon-Nash_fully}
Let $\phi_1^{h,k}, \dots,\phi_N^{h,k}$ the piecewise affine functions defined in \eqref{piecewise_linear_epsilon_Nash}
Let us denote by $L_d$ the constant defined by
\begin{equation}\label{lip_dis}
L_d=\max_{1\le j\le n_s,1\le i\le N}\frac{| \phi_i^j-\phi_i^l|} {|x^j-x^l|}.
\end{equation}
where $\phi_i^{h,k}(x^j)=\phi_i^j$.
Let us assume that hypotheses $H1$, $H2$, $H3$ and $H4$ are satisfied.

Let  $\epsilon>0$. There exists $h_0>0, k_0>0$,  with $k_0$ depending on $L_d$, such that for $h\le h_0$, $k\le k_0$   and all $x\in\Omega$, if $(\psi_1,\dots,\psi_n)\in\mathcal{U}$ is a $N$-tuple of arbitrary admissible stationary strategies satisfying \eqref{lipschitz_strategies}, then
\begin{equation}\label{final_bound}
W_i(\phi_i^{h,k},\phi_{-i}^{h,k},x)\ge W_i(\psi_i,\phi_{-i}^{h,k},x)-\epsilon,\quad i=1,\dots, N.
\end{equation}
\end{theorem}
\begin{proof}
The proof of the following theorem is similar to the proof of Theorem  \ref{epsilon-Nash} applying
Proposition \ref{consistency_fully} instead of Proposition \ref{consistency}.

Arguing as in Proposition \ref{consitency_with_order} and using standard interpolation arguments together
with \eqref{lip_dis} it can be proved
\begin{equation}\label{con_29}
|W_{i,h,k}(\phi_{i}^{h,k},\phi_{-i}^{h,k},x)
-\widetilde W_{i,h,k}(\phi_{i}^{h,k},\phi_{-i}^{h,k},x)|\le Ck.
\end{equation}
From Proposition \ref{consistency_fully} and the above inequality,  given $\epsilon>0$, there exist positive constants $h_0$, $k_0$ such that for $h\le h_0$, $k\le k_0$
\begin{eqnarray}\label{both}
\bigl|\widetilde W_{i,h,k}(\phi_{i}^{h,k},\phi_{-i}^{h,k},x)-W_i(\phi_i^{h,k},\phi_{-i}^{h,k},x)\bigr|&\le& \frac{\epsilon}{4}\nonumber\\
|W_{i,h,k}(\phi_{i}^{h,k},\phi_{-i}^{h,k},x),x)
-\widetilde W_{i,h,k}(\phi_{i}^{h,k},\phi_{-i}^{h,k},x)&\le&\frac{\epsilon}{4}.
\end{eqnarray}
Adding and subtracting terms and using Theorem \ref{verification_fully-time_discrete} we can write
\begin{align}\label{drei}
W_i(\phi_i^{h,k},\phi_{-i}^{h,k},x)&=(W_i(\phi_i^{h,k},\phi_{-i}^{h,k},x)-W_{i,h,k}(\phi_i^{h,k},\phi_{-i}^{h,k},x))+W_{i,h,k}(\phi_i^{h,k},\phi_{-i}^{h,k},x)\nonumber\\
&\ge(W_i(\phi_i^{h,k},\phi_{-i}^{h,k},x)-W_{i,h,k}(\phi_i^{h,k},\phi_{-i}^{h,k},x))+W_{i,h,k}(\psi_i,\phi_{-i}^{h,k},x)\nonumber\\
&=W_i(\psi_i,\phi_{-i}^{h,k},x)+(W_i(\phi_i^{h,k},\phi_{-i}^{h,k},x)-W_{i,h,k}(\phi_i^{h,k},\phi_{-i}^{h,k},x))\nonumber\\
&\quad +(W_{i,h,k}(\psi_i,\phi_{-i}^{h,k},x)-W_i(\psi_i,\phi_{-i}^{h,k},x)),
\end{align}
where $(\psi_1,\dots,\psi_N)\in\mathcal{U}$ is an $N$-tuple of admissible strategies satisfying (\ref{lipschitz_strategies}).
We now observe that applying \eqref{both}
\begin{align*}
|W_i(\phi_i^{h,k},\phi_{-i}^{h,k},x)-W_{i,h,k}(\phi_i^{h,k},&\phi_{-i}^{h,k},x))|\le
|W_{i}(\phi_i^{h,k},\phi_{-i}^{h,k},x)-\widetilde W_{i,h,k}(\phi_i^{h,k},\phi_{-i}^{h,k},x)|
\\
&+|\widetilde W_{i,h,k}(\phi_i^{h,k},\phi_{-i}^{h,k},x)-W_{i,h,k}(\phi_i^{h,k},\phi_{-i}^{h,k},x)|\le \frac{\epsilon}{2}
\end{align*}
The same argument can be applied to the last term in \eqref{drei} to conclude \eqref{final_bound}.
\end{proof}
Next theorem is analogous to Theorem \ref{th4} and its proof is similar but applying  Proposition \ref{consitency_with_order_fully}
instead of Proposition \ref{consitency_with_order} and arguing as in the previous theorem.
\begin{theorem}
Let assumptions of Theorem 6 hold. Furthermore, let us assume that $\rho>L$ with
$L=L_g(1+NL_s)$. There exists $h_0>0, k_0>0$ such that for $h\le h_0$, $k\le k_0$  and all $x\in\mathbb{V}$, if $(\psi_1,\dots,\psi_n)\in\mathcal{U}$ is a $N$-tuple of arbitrary admissible stationary strategies satisfying \eqref{lipschitz_strategies}, then
\begin{equation*}
W_i(\phi_i^{h,k},\phi_{-i}^{h,k},x)\ge W_i(\psi_i,\phi_{-i}^{h,k},x)-C(h+k),\quad i=1,\dots, N.
\end{equation*}\end{theorem}
%\begin{remark} Aqu\'{\i} se puede decir que se pueden debilitar las condiciones de regularidad de los controles
%como se hace en \cite{Javier_yo} y aun obtener los mismos resultados.
%\textcolor{red}{Comprobar las pruebas
%de las proposiciones y dem\'as porque se usa $|\cdot|$ indistintamente para funciones escalares y vectoriales
%y en el segundo caso no est\'a claro que es eso}.
%\end{remark}
\begin{remark}
Let us observe that since we have a finite number of elements $1\le j\le n_s$, $1\le i\le N$ in \eqref{lip_dis}
$L_d$ is always a finite constant. However, the constant in inequality \eqref{con_29} depends linearly
in $L_d$, which means that the size of $L_d$ influences the size of the error. In particular, a bigger constant will produce a bigger error. Let us also observe that the value of $L_d$ can always be computed in practice
giving an indicator on the size of the error. To get good numerical results one expects that the value of $L_d$ remains bounded as $h$, $k$ goes to $0$. However, even a value of $L_d$ not bounded for example in $k$, let say $L_d=O(k^{-\alpha})$ with $0<\alpha<1$ would lead to a bound of size
$k^{1-\alpha}$ in \eqref{con_29} which still allows convergence of the method although loosing  first order in time.
\end{remark}
\section{Numerical Experiments}

%\color{red}

The numerical method employed in the experiments to compute the fully discrete solution will be the Value-Iteration based method developed in \cite[Appendix C]{DeFrutos2016}, which corresponds to a piecewise monotone interpolation method.

For a particular game, let $\phi^{h,k}_{i}$ denote the numerical solution obtained with this numerical method for a particular choice of $h,k>0$.

The objective is to check that, given $\epsilon>0$, for $h,k$ small enough and all $\psi_i$ admisible
\begin{equation}\label{main}
W_1(\phi_i^{h,k},\phi_{-i}^{h,k},x)\ge W_1(\psi_i,\phi_{-i}^{h,k},x)-\epsilon.
\end{equation}
For simplicity in the notation, assume that the game is played just by two players (the same argument can be directly extended to any number of players $N$). Furthermore, let us assume that player 2 always plays $\phi_{2}^{h,k}$.

Let $\Psi_1^{h,k}$ be the {best possible response} in the continuous game for player 1 to $\phi_{2}^{h,k}$. Since the strategy of player 2 is fixed, this best response corresponds to the solution of an optimal control problem.

For all $\psi_1$ admissible, it holds that
 $$W_1(\Psi_1^{h,k},\phi_{2}^{h,k},x)\ge W_i(\psi_1,\phi_{2}^{h,k},x),$$
and if
\begin{equation}\label{con2}
W_1(\phi_1^{h,k},\phi_{2}^{h,k},x)\ge W_1(\Psi_1^{h,k},\phi_{2}^{h,k},x)-\epsilon
\end{equation}
then we have for all admissible $\psi_1$
\begin{equation*}
\begin{split}
W_1(\phi_1^{h,k},\phi_2^{h,k},x)&\ge W_1(\Psi_1^{h,k},\phi_{2}^{h,k},x)-\epsilon\\
&\ge W_1(\psi_1,\phi_{2}^{h,k},x)-\epsilon.
\end{split}
\end{equation*}
Condition \eqref{con2}, and consequently condition \eqref{main}, follows if
\begin{equation}\label{equivcon2}
\lim_{h,k\rightarrow0}\max_{x}\bigl\{W_1(\Psi_1^{h,k},\phi_{2}^{h,k},x)-W_i(\phi_1^{h,k},\phi_2^{h,k},x)\bigr\}=0.
\end{equation}

In the case that an analytical best response solution cannot be computed, in order to approximate it for a particular choice of $h,k>0$, a numerical approximation $\Psi_1^{\Delta h,\Delta k}$ is computed, where $\Delta h$ and $\Delta k$ denote small enough time and spatial discretizations.

The value functions $W_i(\phi_1^{h,k},\phi_2^{h,k},x)$ (respectively $W_i(\Psi_1^{\Delta h,\Delta k},\phi_2^{h,k},x)$) are computed numerically approximating, with high precision
\begin{equation*}
\int_0^\infty f_i(x,\phi_1^{h,k},\phi_2^{h,k})e^{-\rho t} dt,
\end{equation*}
subject to:
\begin{equation*}
\dot{x}=g(x,\phi_1^{h,k},\phi_2^{h,k}),\quad x(0)=x_0.
\end{equation*}

\subsection{A transboundary pollution problem}
To illustrate the theoretical results we have chosen a model problem from \cite{DeFrutos2016},  where a transboundary pollution differential game is analyzed.

Each player $i, \ i\in\{1,...,N\}$, corresponds to a country, which has a pollution stock $p_i\geq 0, \ i\in\{1,...,N\}$ (state variables) and can control its own level of emissions $v_i\geq 0, \ i\in\{1,...,N\}$. Pollution may travel across countries, subject to a predetermined spatial relationships among them.

In particular, for simplicity, we have chosen the first example of a two players game presented in \cite{DeFrutos2016}. In  Figure 1, $\Omega_i, \ i\in\{1,2\}$ represents country $i$ and the spatial relationship between both countries allows pollution to travel among them according to Flick's Law.

It is important to note that Figure 1 is just a representation of the spatial relationship between countries. We
refer to \cite[Appendix B]{DeFrutos2016} for details. We denote by $p_i$ the average stock of pollution over the whole region $\Omega_i$, and by $v_i$ the averaged emissions over $\Omega_i$. The stock of pollution $[p_1,p_2]^T$ are the state variables and player $i$ controls the emission $v_i$ over $\Omega_i$. Therefore, the problem has two
 scalar state variables and one control variable per country.

The objective of player $i$, $i=1,2,$ is to find $v_i$ that maximizes
\begin{eqnarray}\label{w_exp}
W_i(v_i,v_{-i},p_{0,1},p_{0,2})=\int_0^\infty e^{-\rho t}\Big(v_i\big(A-\frac{v_i}{2}\big)-\frac{\varphi}{2}p_i^2\Big) \, dt,
\end{eqnarray}
subject to:
\begin{eqnarray}\label{edo_exp}
\begin{bmatrix} \dot p_1\\ \dot p_2 \end{bmatrix}=
\nu \begin{bmatrix} -1&\phantom{-}1\\\phantom{-}1&-1 \end{bmatrix}
\begin{bmatrix} p_1\\p_2 \end{bmatrix}-c\begin{bmatrix} p_1\\p_2 \end{bmatrix}
+\beta\begin{bmatrix}
v_1\\v_2
\end{bmatrix}.
\end{eqnarray}
\begin{figure}\label{fig_1}
\begin{center}
\includegraphics[width=7cm,height=4cm]{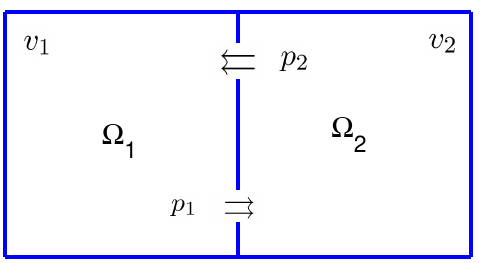}
\end{center}
\caption{Model problem}
\end{figure}
In \eqref{w_exp}-\eqref{edo_exp}, $A$, $\rho$, $\varphi$, $\nu$, $c$ and $\beta$ are constants, see \cite[Appendix B]{DeFrutos2016} for details.
In the numerical experiments of this subsection, the values of the parameters are $\varphi=1$, $A=0.5$, $\rho=0.01$ $\beta=1$, $c=0.5$.  Emissions must be postive and the Nash-equilibrium strategies correspond to a piece-wise affine function, that, as it is well-known, could be computed solving a set of Ricatti equations, see \cite{DeFrutos2016}.

Figure 2 represents the value function $W_1$ (left) for $(p_1,p_2)\in[0,2]\times[0,2]$ and the feed-back strategies $\phi_1$ (center) and $\phi_2$ (right). Since, both players are symmetric, so are their feed-back strategies.
\begin{figure}\label{fig_2}
\begin{center}
\includegraphics[width=15cm,height=4cm]{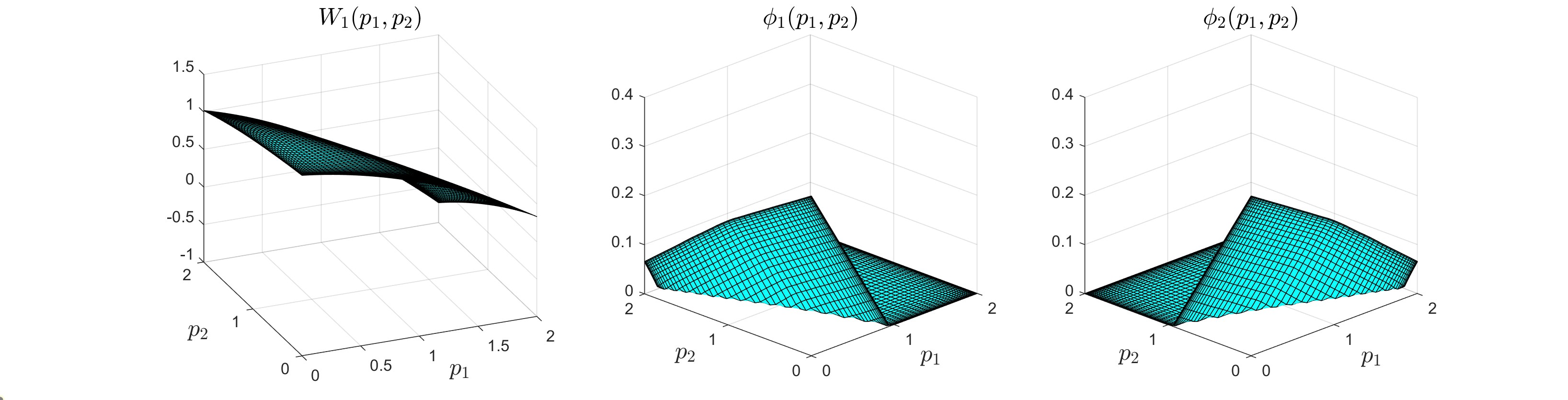}
\end{center}
\caption{Value Function $W_1$(left), and feed-back strategies $\phi_1$ (center), $\phi_2$ (right)}
\end{figure}
Let $h,k>0$ and $x^j=(p_{1,j},p_{2,j}), \ j=1,...,n^k_s$ the (equally-spaced) spatial discretization induced by $k$. The discrete Bellman equations of the fully discrete  problem are
\begin{equation*}
\begin{split}
V_{i,h,k}&=\max_{v_i^j\ge 0}\Big\{{h}\Big(v_i^j\big(A-\frac{v_i^j}{2}\big)-\frac{\varphi}{2}p_{i,j}^2\Big) \\
&+(1-\rho h)V_{i,h,k}\Big(\begin{bmatrix} p_{1,j}\\p_{2,j} \end{bmatrix}
+h\big( \nu \begin{bmatrix} -1&\phantom{-}1\\\phantom{-}1&-1 \end{bmatrix}
\begin{bmatrix} p_{1,j}\\p_{2,j} \end{bmatrix}-c\begin{bmatrix} p_{1,j}\\p_{2,j} \end{bmatrix}
+\beta\begin{bmatrix}
v_1^j\\v_2^j
\end{bmatrix}\big) \Big){\Big\}},
\end{split}
\end{equation*}
where $v_i^j$ represents the fully discrete approximation to $v_i(x^j)$.

First, we analyze (\ref{equivcon2}) with respect to the temporal discretization, taking a value $k$ small enough such that the spatial error discretization is negligible.

In Figure 3 we represent $W_1(\Psi_1^{\Delta h,\Delta k},\phi_{2}^{h,k},p_{01},p_{02})-W_1(\phi_1^{h,k},\phi_2^{h,k},p_{01},p_{02})$
for $h=1/4,1/8,1/16$ and $1/32$. As it can be seen, the best response $\Psi_1^{\Delta h,\Delta k}$ gives a better result than playing $\phi_1^{h,k}$ for any value $h$ and this difference decreases as $h\rightarrow 0$.
\begin{figure}
\begin{center}\label{fig_3}
\includegraphics[width=18cm]{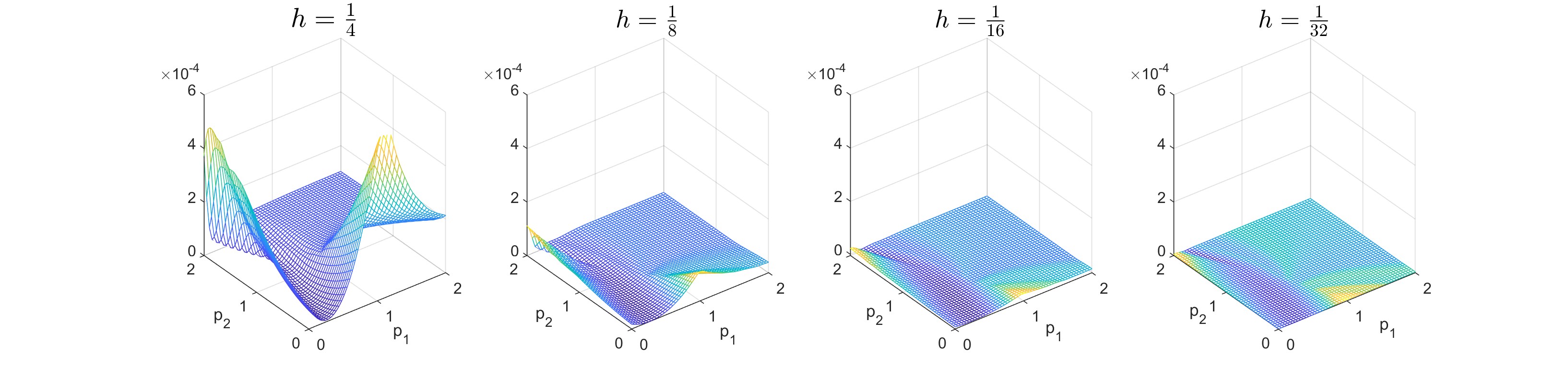}
\end{center}
\caption{$W_1(\Psi_1^{\Delta h,\Delta k},\phi_{2}^{h,k},p_{01},p_{02})-W_1(\phi_1^{h,k},\phi_2^{h,k},p_{01},p_{02})$ for $h=1/4, 1/8, 1/16$ and $1/32$}
\end{figure}

In Figure 4 we represent $\displaystyle\max_{p_{01},p_{02}}\bigl\{W_1(\Psi_1^{\Delta h,\Delta k},\phi_{2}^{h,k},p_{01},p_{02})-W_1(\phi_1^{h,k},\phi_2^{h,k},p_{01},p_{02})\bigr\}$ for the same values of $h$.
%The slope of the line is around $2$, so that we can observed a faster rate of convergence than the one predicted by %the theory in this concrete example.
\begin{figure}\label{fig_4}
\begin{center}
\includegraphics[width=8cm]{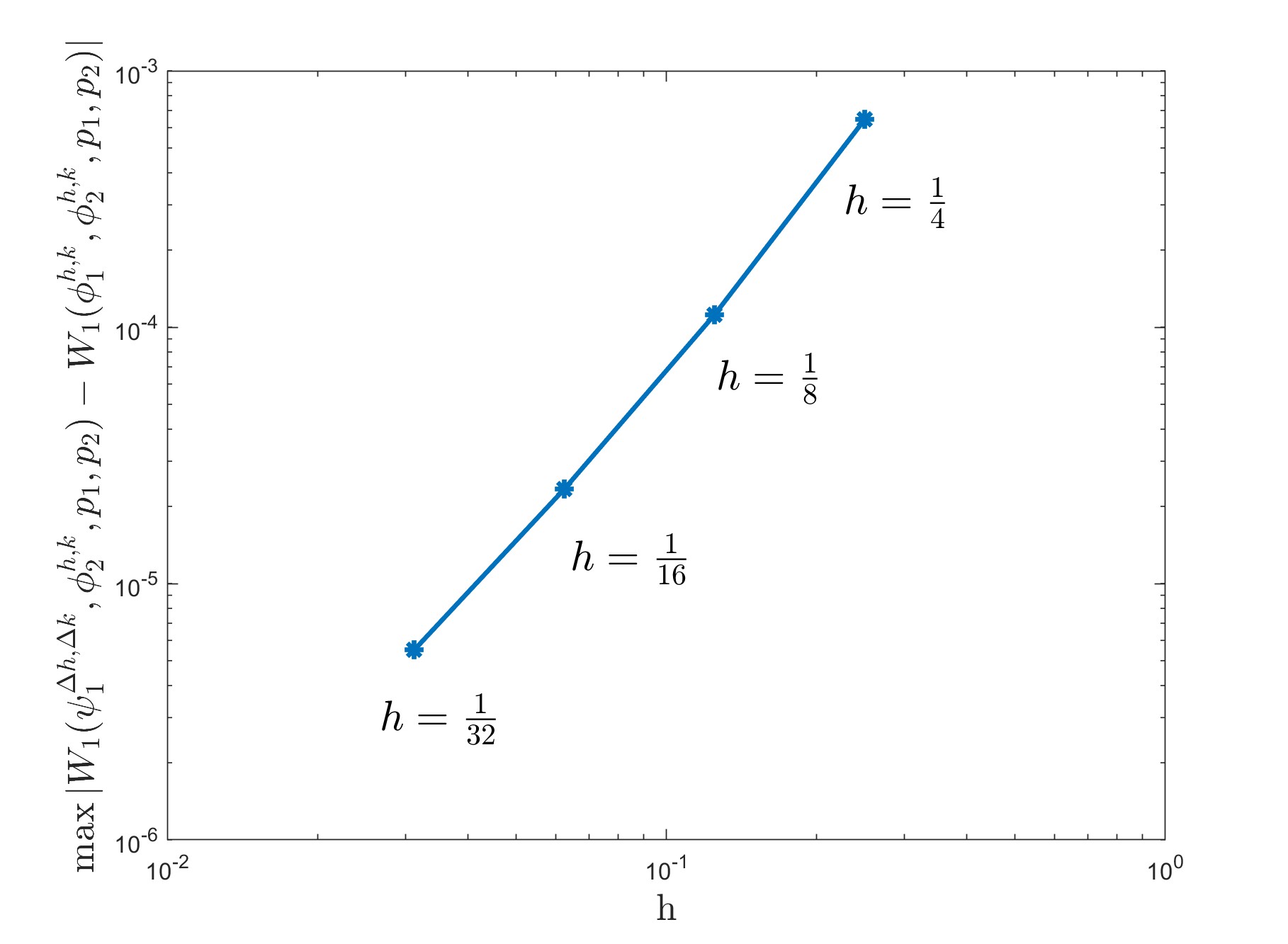}
\end{center}
\caption{$\displaystyle\max_{p_{01},p_{02}}\bigl\{W_1(\Psi_1^{\Delta h,\Delta k},\phi_{2}^{h,k},p_{01},p_{02})-W_1(\phi_1^{h,k},\phi_2^{h,k},p_{01},p_{02})\bigr\}$ for $h=1/4, h=1/8, h=1/16$ and $h=1/32$.}
\end{figure}
In this particular example, we also have convergence of strategies as it can be seen in Figure 5, where we represent $\displaystyle\max_{p_{01},p_{02}}\bigl\{\Psi_1^{\Delta h, \Delta k}- \phi_1^{h,k}\bigr\}$. The slope of the line is around $1$, giving first order of convergence.
\begin{figure}\label{fig_5}
\begin{center}
\includegraphics[width=8cm]{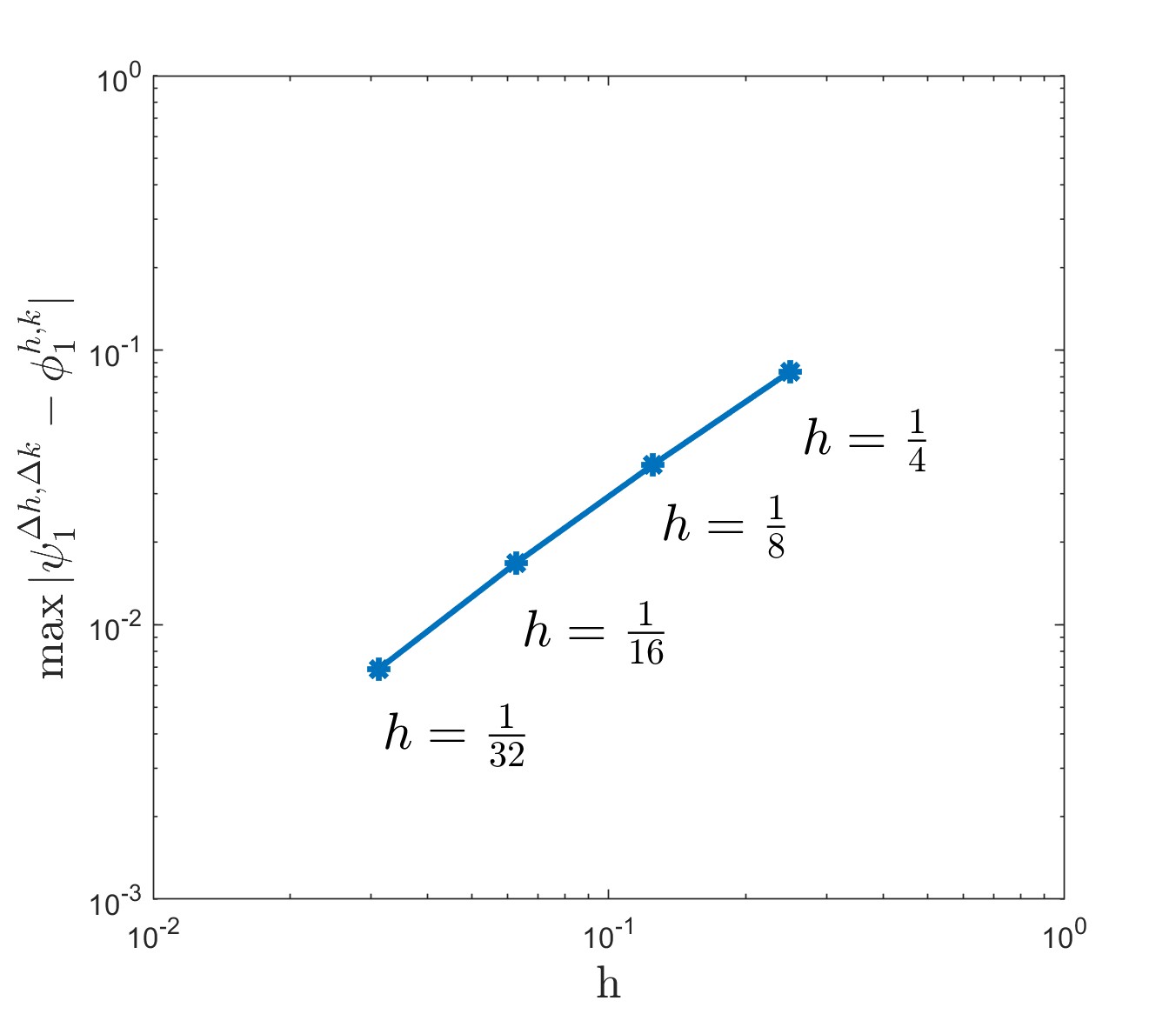}
\end{center}
\caption{$\displaystyle \max_{p_1,p_2} | \Psi_1^{dh,dk}(p_1,p_2)-\phi_{1}^{h,k}(p_1,p_2)|$
for $h=1/4, 1/8, 1/16$ and $1/32$}
\end{figure}

To analyze (\ref{equivcon2}) with respect to the spatial discretization, we take a value $h$ small enough, such that the temporal error discretization is negligible.
In Figure 6 we represent $W_1(\Psi_1^{\Delta h,\Delta k},\phi_{2}^{h,k},p_{01},p_{02})-W_1(\phi_1^{h,k},\phi_2^{h,k},p_{01},p_{02})$
for $k=1/2, 1/4$ and $1/8$. We obtain the same behaviour as in Figure 3. The best response $\Psi_1^{\Delta h,\Delta k}$ gives a better result than playing $\phi_1^{h,k}$ for any value $k$, with a difference that decreases as $k\rightarrow 0$.
\begin{figure}\label{fig_6}
\begin{center}
\includegraphics[width=18cm]{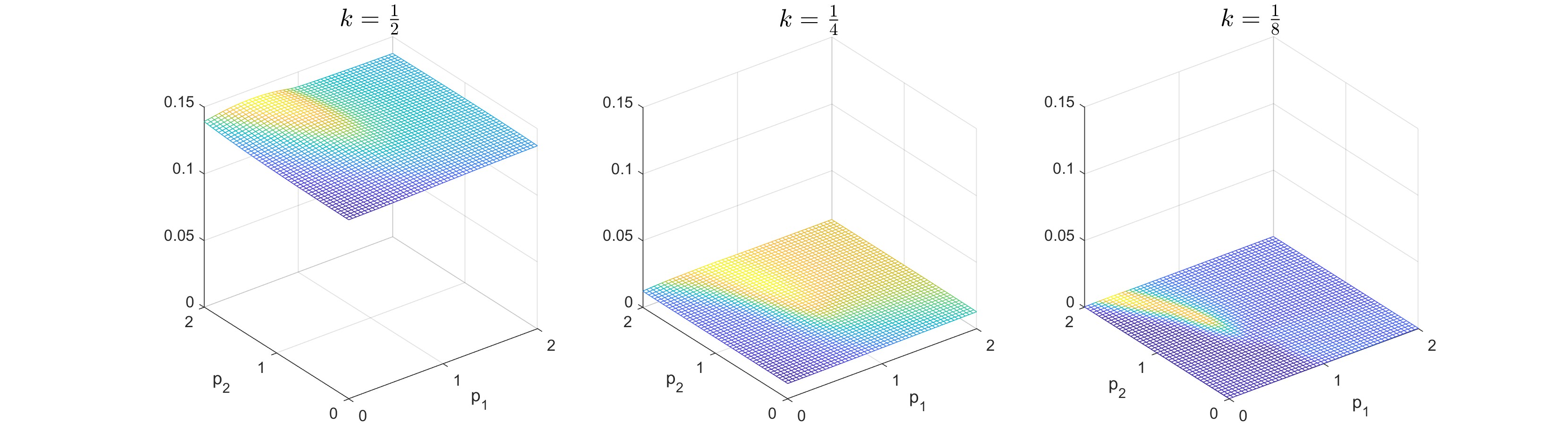}
\end{center}
\caption{$\displaystyle W_1(\Psi_1^{dh,dk},\phi_{2}^{h,k},p_{01},p_{02})-W_1(\phi_1^{h,k},\phi_2^{h,k},p_{01},p_{02})$ for $k=1/2, k=1/4$ and $k=1/8$.}
\end{figure}
Figure 7 represents $\displaystyle\max_{p_{01},p_{02}}\bigl\{W_1(\Psi_i^{\Delta h, \Delta k},\phi_{-i}^h,p_{01},p_{02})-W_1(\phi_1^h,\phi_2^h,p_{01},p_{02})\bigr\}$  (left) and $\displaystyle\max_{p_{01},p_{02}}\bigl\{\Psi_1^{\Delta h, \Delta k}- \phi_1^{h,k}\bigr\}$ (right) for $k=1/2, 1/4$ and $1/8$. We obtain a very fast convergence towards 0 in the value function and first order of convergence in the strategies.
\begin{figure}\label{fig_7}
\begin{center}
\includegraphics[width=12cm]{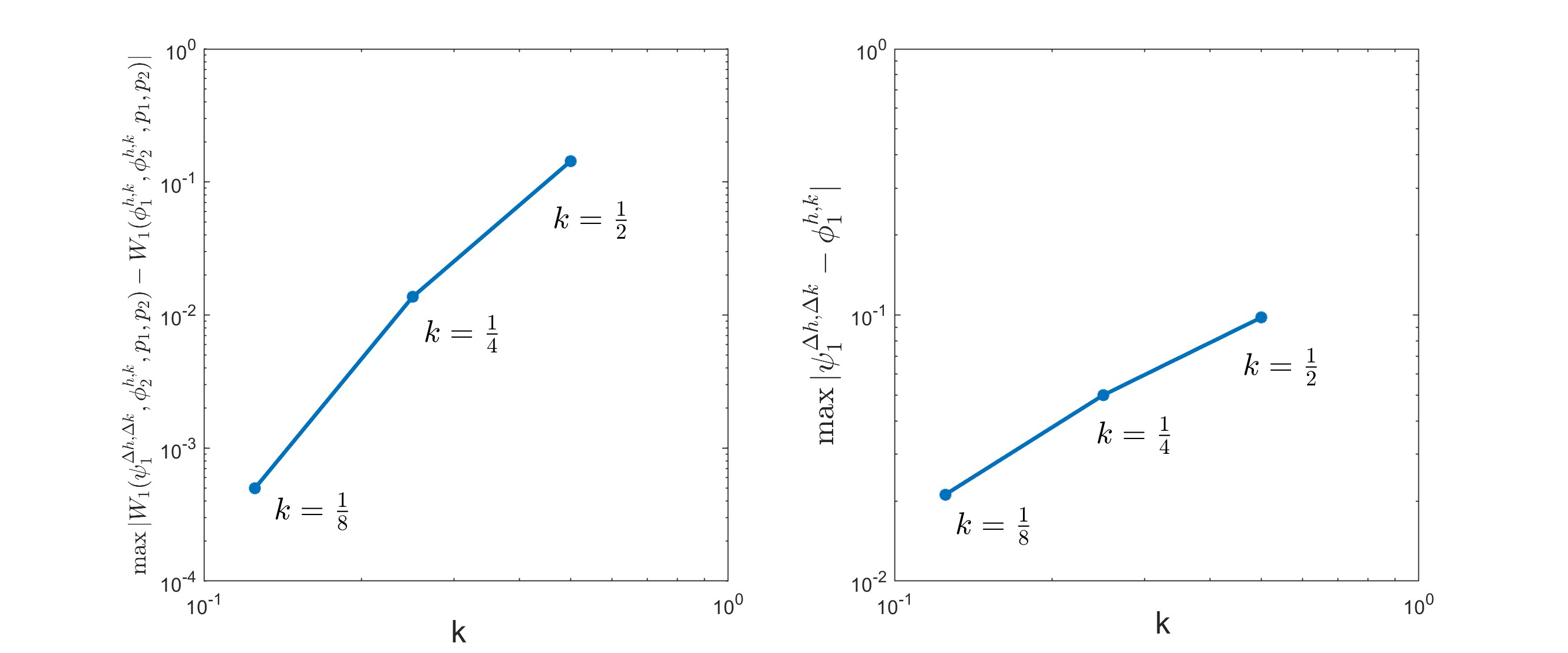}
\end{center}
\caption{$\displaystyle\max_{p_{01},p_{02}}\bigl\{W_1(\Psi_1^{dh,dk},\phi_{2}^{h,k},p_{01},p_{02})-W_1(\phi_1^{h,k},\phi_2^{h,k},p_{01},p_{02})\bigr\}$ (left) and $\displaystyle \max_{p_1,p_2} | \Psi_1^{dh,dk}(p_1,p_2)-\phi_{1}^{h,k}(p_1,p_2)|$ for $k=1/2, k=1/4$ and $k=1/8$.}
\end{figure}
\subsection{A Lanchester type differential game}
In order to check the performance of the method in a more complicated game not being linear-quadratic, we have chosen a Lanchester-type differential game where the objective of player $1$ is to find a strategy that maximizes
\begin{eqnarray}\label{w_lan}
\begin{aligned}
W_1(v_1,v_{2},x_0)&=\int_0^\infty e^{-\rho t}\Big(\pi_1x-\varphi v_1^{1{.}5}\Big) \, dt,\\
\end{aligned}
\end{eqnarray}
the objective of player $2$ is to maximize
\begin{eqnarray}\label{w_lan2}
\begin{aligned}
W_2(v_1,v_{2},x_0)&=\int_0^\infty e^{-\rho t}\Big(\pi_2(1-x)-\varphi v_2^{1{.}5}\Big) \, dt, \\
\end{aligned}
\end{eqnarray}
both subject to:
\begin{eqnarray}\label{edo_lan}
\frac{dx}{dt}=(1-x)v_1-xv_2.
\end{eqnarray}
Lanchester type models have been employed, for example, in competitive advertising decisions, where the state variable $x\in[0,1]$ represents the share of the market of player 1 ($1-x$ than that of player 2) and $v_i\geq 0, \ i\in\{1,2\}$ are the amounts inverted in advertising, see \cite{georges}.

In Figure 8 we represent a numerical approximation to $W_1$ (left) and $\phi_1$ (right) for $\pi_1=\pi_2=\varphi=1$.
\begin{figure}\label{fig_8}
\begin{center}
\includegraphics[width=12cm]{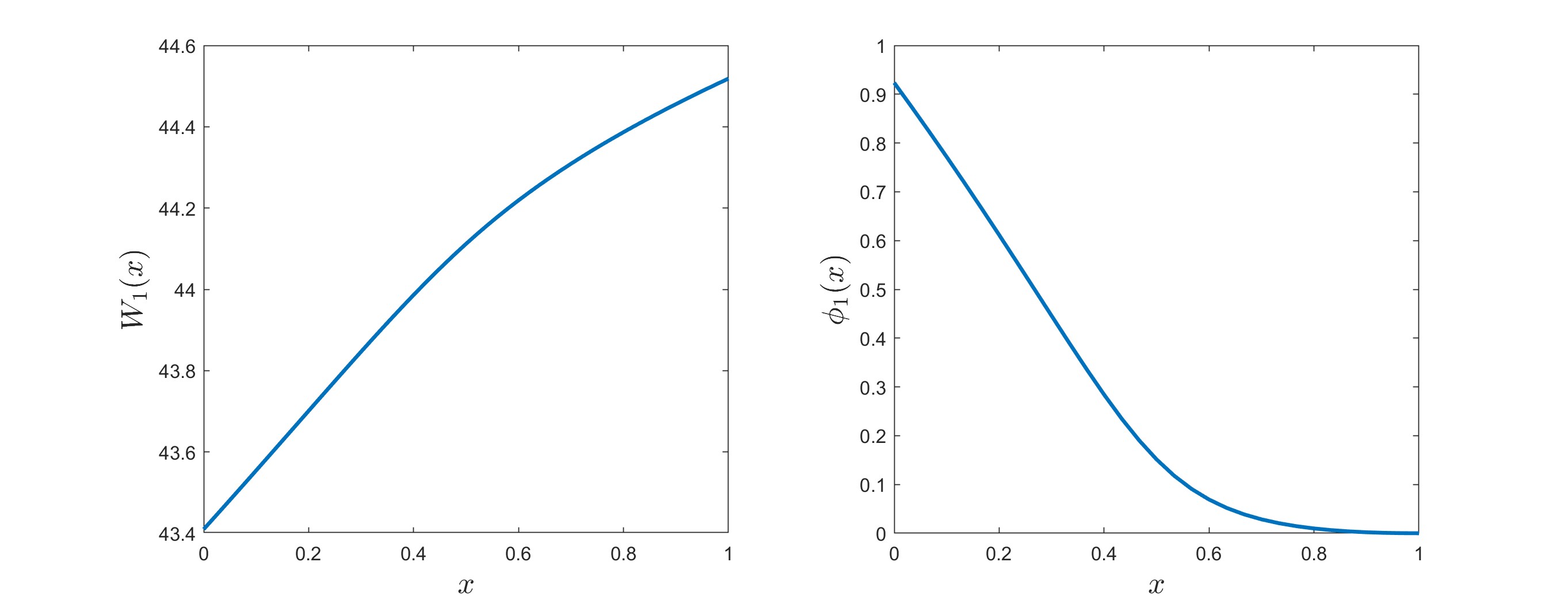}
\end{center}
\caption{Estimated Value Function $W_1(x)$ (left) and feedback strategy $\phi_1(x)$ (right) in the Lancaster model.}
\end{figure}
We proceed as in the previous numerical example. With respect to the temporal discretization, Figure 9  represents $W_1(\Psi_1^{\Delta h,\Delta k},\phi_{2}^{h,k},x_0)-W_1(\phi_1^{h,k},\phi_2^{h,k},x_0)$ for $h=1/2, 1/4, 1/8$ and $1/16$. As it can be seen, the best response $\Psi_1^{\Delta h,\Delta k}$ gives, as expected, a better result than playing $\phi_1^{h,k}$ for any value $h$, and this difference decreases as $h\rightarrow 0$.
\begin{figure}\label{fig_9}
\begin{center}
\includegraphics[width=10cm]{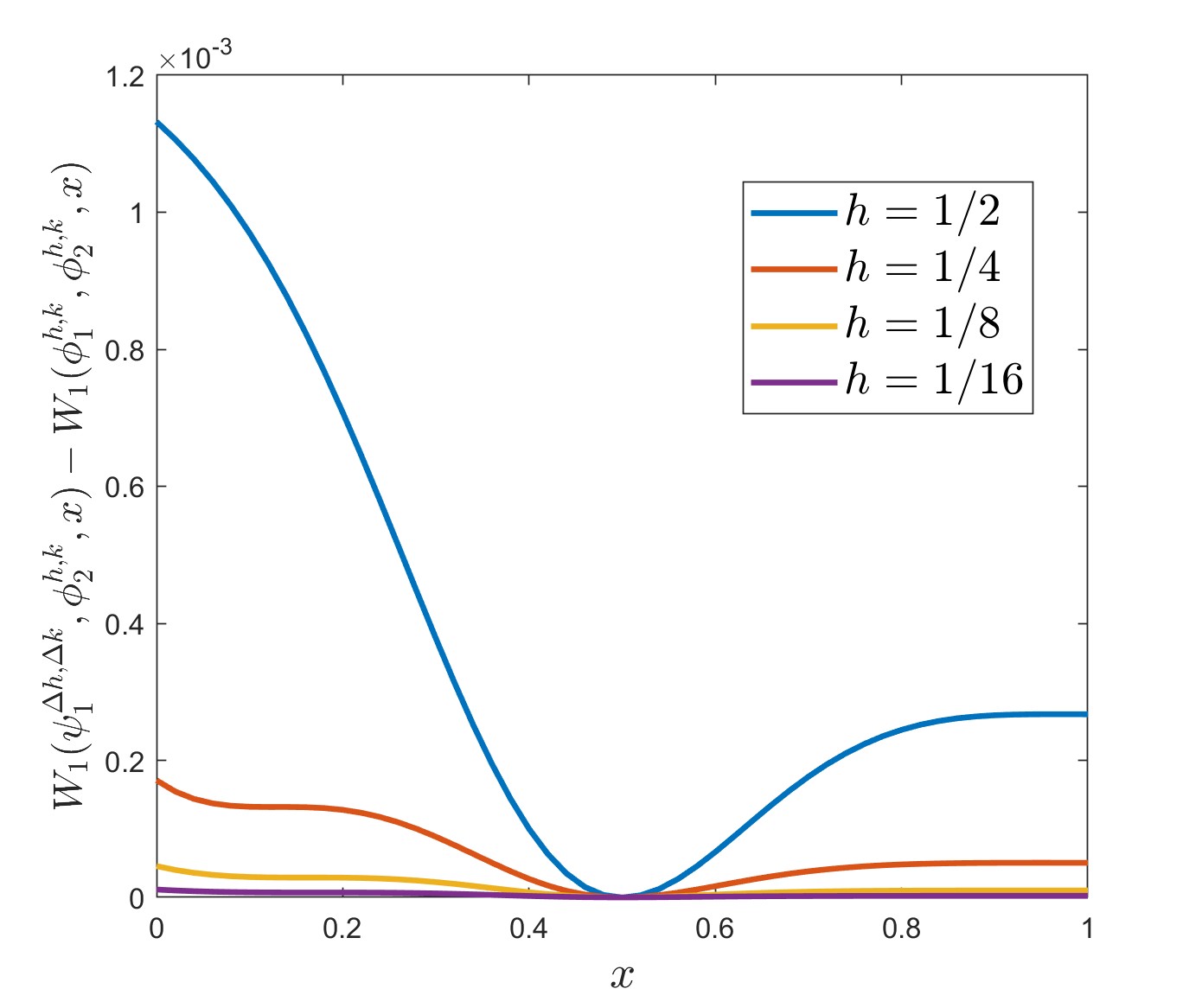}
\end{center}
\caption{$W_1(\Psi_1^{\Delta h,\Delta k},\phi_{2}^{h,k},x)-W_1(\phi_1^{h,k},\phi_2^{h,k},x)$ for $h=1/2, h=1/4, h=1/8$ and $h=1/16$}
\end{figure}

Figure 10 represents the error in the value function on the left and in the strategies on the right,
$\displaystyle\max_{x}\bigl\{W_1(\Psi_i^{\Delta h, \Delta k},\phi_{-i}^h,x)-W_1(\phi_1^h,\phi_2^h,x)\bigr\}$ and $\displaystyle\max_{x}\bigl\{\Psi_1^{\Delta h, \Delta k}- \phi_1^{h,k}\bigr\}$, for $h=1/2, 1/4, 1/8$ and $1/16$. We can observe the convergence both in the value function and in the strategies.
\begin{figure}\label{fig_10}
\begin{center}
\includegraphics[width=12cm]{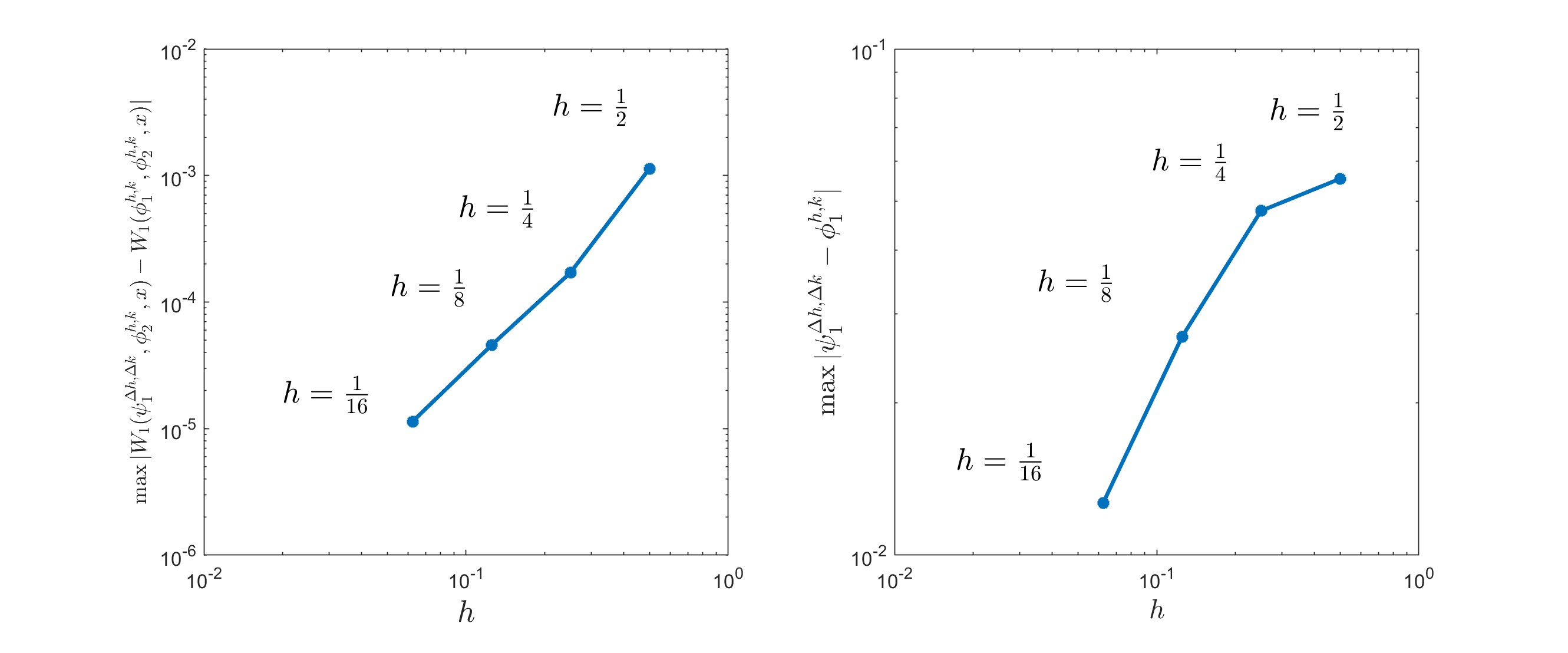}
\end{center}
\caption{$\displaystyle\max_{x}\bigl\{W_1(\Psi_1^{\Delta h,\Delta k},\phi_{2}^{h,k},x)-W_1(\phi_1^{h,k},\phi_2^{h,k},x)\bigr\}$ (left) and $\displaystyle \max_{x} | \Psi_1^{\Delta h,\Delta k}(x)-\phi_{1}^{h,k}(x)|$ for $h=1/2, h=1/4, h=1/8$ and $h=1/16$.}
\end{figure}

In Figure 11 we represent $W_1(\Psi_1^{\Delta h,\Delta k},\phi_{2}^{h,k},x)-W_1(\phi_1^{h,k},\phi_2^{h,k},x)$
for $k=1, k=1/2$ and $k=1/4$. The best response $\Psi_1^{\Delta h,\Delta k}$ gives a better result than playing $\phi_1^{h,k}$ for any value $k$, with the difference decreasing as $k\rightarrow 0$.
\begin{figure}\label{fig_11}
\begin{center}
\includegraphics[width=8cm]{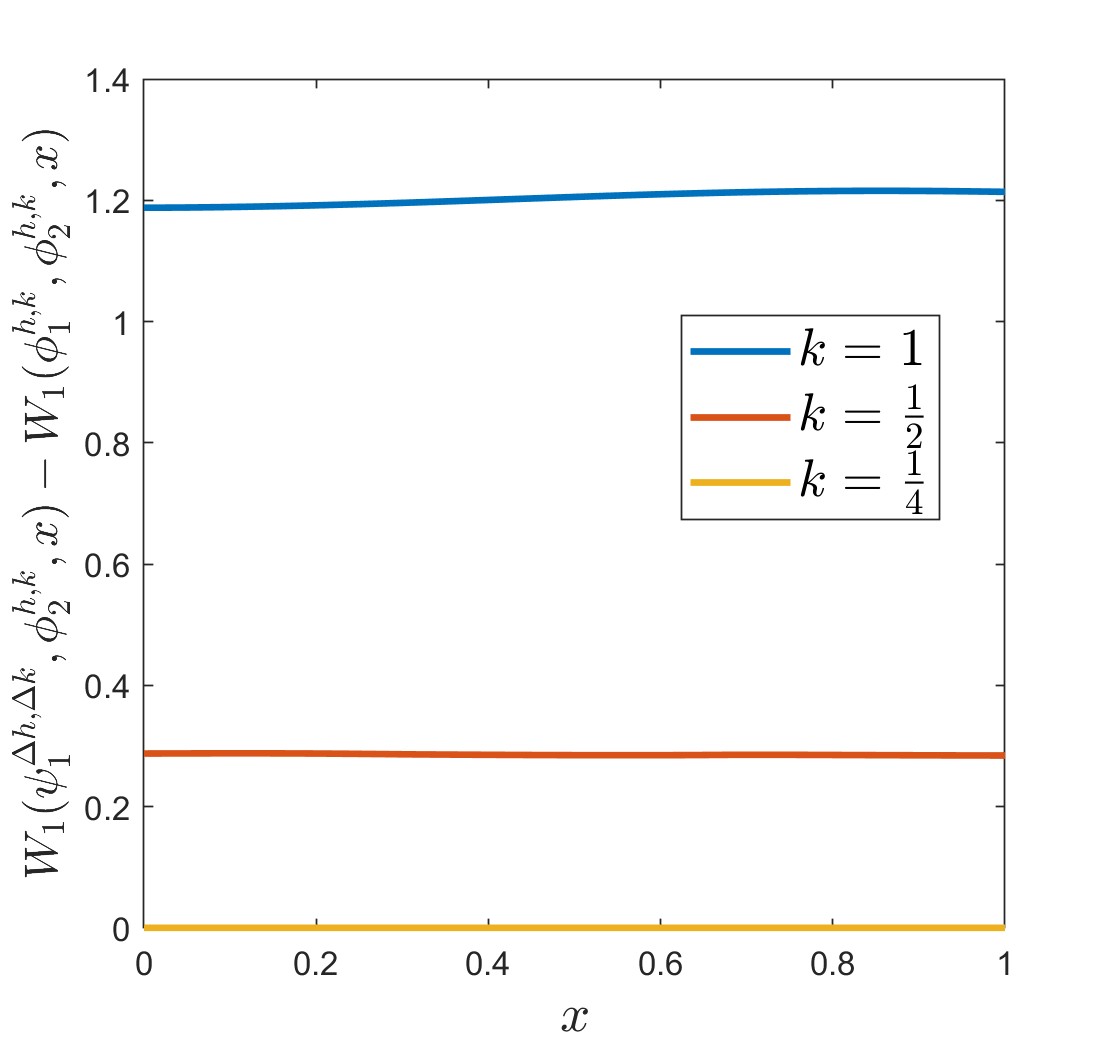}
\end{center}
\caption{$W_1(\Psi_1^{\Delta h,\Delta k},\phi_{2}^{h,k},x)-W_1(\phi_1^{h,k},\phi_2^{h,k},x)$ for $k=1, k=1/2$ and $k=1/4$.}
\end{figure}
Finally, Figure 12 represents $\displaystyle\max_{p_{01},p_{02}}\bigl\{W_1(\Psi_i^{\Delta h, \Delta k},\phi_{-i}^h,x)-W_1(\phi_1^h,\phi_2^h,x)\bigr\}$  (left) and $\displaystyle\max_{x}\bigl\{\Psi_1^{\Delta h, \Delta k}- \phi_1^{h,k}\bigr\}$ (right) for $k=1, k=1/2$ and $k=1/4$. Again, we obtain a very fast convergence towards 0 in the Value Function and first order of convergence in the strategies.
\begin{figure}\label{fig_12}
\begin{center}
\includegraphics[width=10cm]{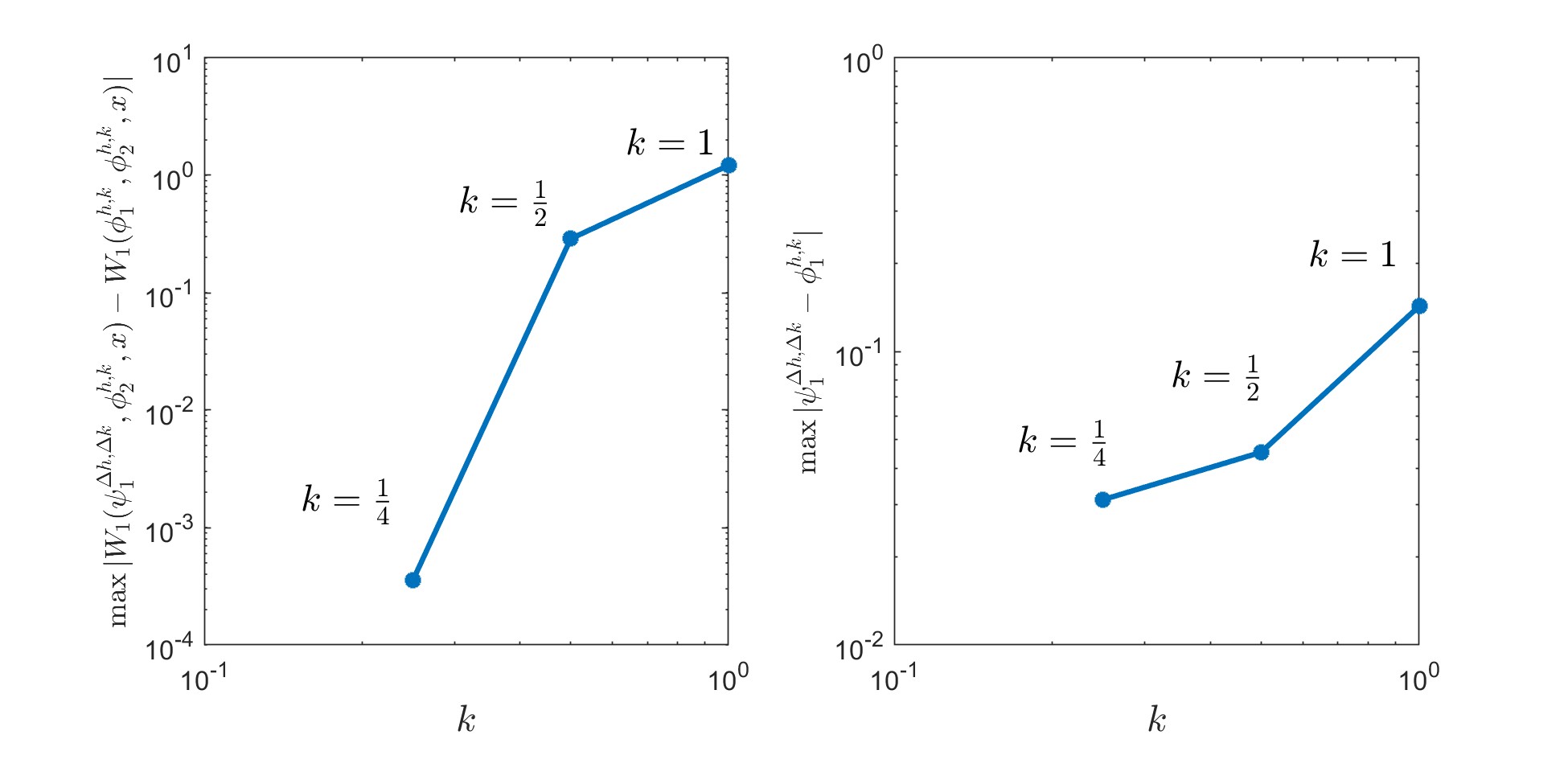}
\end{center}
\caption{$\displaystyle\max_{x}\bigl\{W_1(\Psi_1^{\Delta h,\Delta k},\phi_{2}^{h,k},x)-W_1(\phi_1^{h,k},\phi_2^{h,k},x)\bigr\}$ (left) and $\displaystyle \max_{x} | \Psi_1^{\Delta h,\Delta k}(x)-\phi_{1}^{h,k}(x)|$ for $k=1, k=1/2$ and $k=1/4$.}
\end{figure}
\section{Concluding remarks}

In this paper we analyze a semilagrangian approach to numerically approximate  Markovian Nash equilibria of differential games. We prove that markovian Nash equilibria of the discrete-time and fully discrete approximations, respectively, are $\epsilon$-Nash equilibria of the differential game with $\epsilon$ arbitrarily small for $h$ (the discretization time step) or $h$ and $k$ (the
discretization time step and the spatial mesh size) small enough. Under some restrictive hypotheses we prove that $\epsilon=\mathcal{O}(h)$
(respectively $\epsilon=\mathcal{O}(h+k))$. Although the hypotheses can be seen as too exigent, they often apply in the applications, particularly when a bounded domain, positively invariant for the flow of the dynamics, containing the region of interest can be identified, see \cite{DeFrutos2018} for an example.
\bigskip

\noindent {\bf Funding}  This research has been partially supported by projects   PID2022-136550NB-I00 funded by MICIU/AEI and FEDER (EU) (Javier de Frutos and Julia Novo) and TED2021-130390B-I00 funded by MICIU/AEI/10.13039/501100011033 and by European Union Next Generation EU/PRTR (Javier de Frutos and V\'{\i}ctor Gat\'{o}n)
and PID2024-155429NB-I00 funded by MICIU/AEI/10.13039/501100011033 and the European Union Next Generation EU/PRTR (V\'ictor Gat\'on).
\bigskip

\noindent {\bf Declarations}

{\bf Conflict of interest} The authors have no relevant financial or non-financial interests to disclose.


\bigskip

\begin{thebibliography}{10}

\bibitem{Akian} { M. Akian, M., Gaubert, S. \& Lakhoua, A.},
{\em The max-plus finite element method for solving deterministic optimal control problems: basic properties and convergence analysis}, SIAM J. Control Optim. 47, 2008, pp. 817--848.

\bibitem{Bardi}
Bardi, M. \&  Capuzzo-Dolcetta, I., {\em Optimal Control and Viscosity Solutions of Hamilton-Jacobi-Belmann Equations.}
 {Springer Science+Business Media, LLC, New York, 1997.}

\bibitem{Basar} Ba\c{s}ar T. \& Olsder G.J., {\em Dynamic noncooperative game theory}. SIAM, Philadelphia, (1999).

\bibitem{Basar_et_al} Ba\c{s}ar T., Haurie A. \& Zaccour, G., {\em Nonzero-Sum Differential Games}, in Handbook of Differential Games, Ba\c{s}ar T. and Zaccour, G., eds., Springer, 2018.

\bibitem{Bo_et_al} { Bokanowski, O.,. Garcke, J.,   Griebel, M.  \&  Klompmaker, I.},
{\em An adaptive sparse grid semi-Lagrangian scheme for first order Hamilton-Jacobi Bellman equations}, J. Sci. Comput. 55, 2013, pp. 575--605.

\bibitem{Capuzzo_paper} Capuzzo-Dolcetta, I, {\em  On a Discrete Approximation of the Hamilton-Jacobi Equation of Dynamic Programming}, Appl. Math. Optim. 10, 1993, 367--377.

\bibitem{carlini} {Carlini, E., Falcone, M. \& Ferretti, R.},
{\em An efficient algorithm for Hamilton-Jacobi equations in high dimension}, Comput. and Visualization in Science, 7, 2004, pp. 15--29.

\bibitem{Carlson} {Carlson, D.A.}, {\em
On the existence of catching-up optimal solutions for Lagrange problems defined on unbounded intervals}, J. Optim. Theory Appl., 49, 1986, pp. 207–225.

\bibitem{Dockner}
 Dockner, E., J{\o}rgensen, S. , Van Long,  N. \& Sorger, G.
{\em Differential Games in Economics and managenement Science.}
 {Cambridge University Press, Cambridge, 2000.}



\bibitem{Falcone} Falcone, M.,
{\em Numerical Solution of Dynamic Programming Equations.}
 {Appendix A in M. Bardi \& I. Capuzzo-Dolcetta. Optimal Control and Viscosity Solutions of Hamilton-Jacobi-Belmann Equations. Springer Science+Business Media, LLC, New York, 1997.}

\bibitem{Falcone_Ferreti} Falcone, M. \&  Ferreti, R.,
{\em Discrete time High-Order schemes for Viscosity Solutions of Hamilton-Jacobi-Bellman Equations}
{Numerische Mathenmatik, 67, 1994, pp. 315--344}

\bibitem{DeFrutos2015}  De Frutos  J. \&  Mart\'{\i}n-Herr\'{a}n  G.,   {\em Does flexibility facilitate sustainability of cooperation over time? A case study from environmental economics}.  Journal of Optimization Theory and Applications 165, 2015, pp. 657--677.

\bibitem{DeFrutos2018} De Frutos, J. \& a Mart\'{\i}n-Herr\'{a}n, G., {\em Selection of a Markov perfect Nash equilibrium in a class of differential games}. Dynamic Games and Applications,  8, 2018, pp. 620--636,

 \bibitem{DeFrutos2016}  De Frutos  J.,   Mart\'{\i}n-Herr\'{a}n  G.,  {\em Spatial effects and strategic behavior in a multiregional transboundary pollution  game}.  Journal of Environmental Economics and Management 97, 2019, pp. 182--207.

\bibitem{Javier_yo} De Frutos, J., Novo J., {\em Optimal bounds for numerical approximations of infinite horizon problems based on dynamic programming approach}, SIAM J. Control Optimization 61, 2023, 415--433.

\bibitem{Guo} { Guo, B.Z. \&  Wu, T.T.},
{\em Approximation of optimal feedback control: a dynamic programming approach}, J. Global Optim. 46, 2010, pp. 395--422.

\bibitem{Haurie_et_al} Haurie,  A., Krawczyk, J.B. \&  Zaccour,  G.,  {\em Games and dynamic games}. World Scientific, Singapore, 2012.

\bibitem{wagener} Jaakkola, N., Wagener, F., {\em Differential games of public investment with an application to climate policy}, CESifo working paper, no 10585, 2023.

\bibitem{georges} J{\o}rgensen, S., Zaccour, G.
{\em Differential Games in Marketing.}
 {Kluwer Academic Publishers, 2004.}

\bibitem{Krawczyk_Petkrov} Krawczyk, J.B. \& V. Petkov, {\em Multistage Games}, in Handbook of Dynamic Game Theory, T. Ba\c{s}ar and G. Zaccour , eds., Springer Nature, 2018.

  \end{thebibliography}
\end{document}